\documentclass[leqno]{amsart}
\usepackage{times}
\usepackage{amsfonts,amssymb,amsmath,amsgen,amsthm}
\usepackage{hyperref}
\usepackage{color}
\newcommand{\msc}[2][2000]{%
  \let\@oldtitle\@title%
  \gdef\@title{\@oldtitle\footnotetext{#1 \emph{Mathematics subject
        classification.} #2}}%
}

\theoremstyle{plain}
\newtheorem{theorem}{Theorem}[section]
\newtheorem{definition}[theorem]{Definition}

\newtheorem{lemma}[theorem]{Lemma}

\newtheorem{proposition}[theorem]{Proposition}

\theoremstyle{remark}
\newtheorem{remark}[theorem]{Remark}

\newtheorem{example}[theorem]{Example}

\def\C{{\mathbf C}}
\def\R{{\mathbf R}}
\def\N{{\mathbf N}}
\def\Z{{\mathbf Z}}
\def\T{{\mathbf T}}
\def\H{{\mathcal H}}
\def\Sch{{\mathcal S}}

\def\F{\mathcal F}

\def\({\left(}
\def\){\right)}
\def\<{\left\langle}
\def\>{\right\rangle}
\def\le{\leqslant}
\def\ge{\geqslant}

\def\Tend#1#2{\mathop{\longrightarrow}\limits_{#1\rightarrow#2}}

\def\d{{\partial}}
\def\eps{\varepsilon}
\def\l{\lambda}

\def\si{{\sigma}}

\DeclareMathOperator{\RE}{Re}
\DeclareMathOperator{\IM}{Im}
\DeclareMathOperator{\diver}{div}

\numberwithin{equation}{section}

\begin{document}

\title[Scattering for 
  NLS  under partial confinement]{Scattering for 
  nonlinear Schr\"odinger equation 
  under partial harmonic confinement}
\author[P. Antonelli]{Paolo Antonelli}
\address{Centro di Ricerca Matematica Ennio De Giorgi\\
Piazza dei Cavalieri, 3\\
56100 Pisa, Italy}
\email{paolo.antonelli@sns.it}
\author[R. Carles]{R\'emi Carles}
\address{CNRS \& Univ. Montpellier~2\\Math\'ematiques
\\CC~051\\34095 Montpellier\\ France}
\email{Remi.Carles@math.cnrs.fr}
\author[J. Drumond Silva]{Jorge Drumond Silva}
\address{Center for Mathematical Analysis, Geometry and Dynamical
  Systems\\ Departamento de Matem\'atica\\ Instituto Superior
  T\'ecnico\\ Universidade de Lisboa\\ Av. Rovisco Pais\\ 1049-001 Lisboa\\Portugal}
\email{jsilva@math.ist.utl.pt}
\begin{abstract}
We consider the nonlinear Schr\"odinger equation under a partial
quadratic confinement. We show that the global dispersion
corresponding to the direction(s) with no potential is enough to prove
global in time Strichartz estimates, from which we infer the existence of wave
operators thanks to suitable vector-fields. Conversely, given an
initial Cauchy datum, the solution is 
global in time and asymptotically free, provided that confinement affects
one spatial direction only. This stems from anisotropic Morawetz
estimates,  involving a marginal of the position density.  
\end{abstract}
\thanks{
This work was supported by the French ANR project SchEq
  (ANR-12-JS01-0005-01). J. Drumond Silva was partially supported by
  the Center for Mathematical Analysis, Geometry and Dynamical
  Systems-LARSys through the Funda\c{c}\~ao para a Ci\^encia e
  Tecnologia (FCT/Portugal) Program POCTI/FEDER} 
\maketitle

\section{Introduction}
\label{sec:intro}

\subsection{Motivation}
\label{sec:motivation}

It is well known that the solution to 
\begin{equation}
  \label{eq:LS}
  i\d_t u +\frac{1}{2}\Delta u = 0,\quad (t,x)\in \R\times \R^d,
\end{equation}
with $u_{\mid t=0}\in \Sch(\R^d)$, has large time dispersive properties which
allow, for instance, to develop scattering theories both in the
presence of a linear perturbation and in a nonlinear setting (see
e.g. \cite{CazCourant,DG,Ginibre,Yafaev}). On the other hand, in the
presence of an (isotropic) harmonic potential,
\begin{equation}
  \label{eq:LSP}
  i\d_t u +\frac{1}{2}\Delta u = \frac{|x|^2}{2}u,\quad (t,x)\in
  \R\times \R^d, 
\end{equation}
the solution is periodic in time. This is a
consequence of the existence of a Hilbert basis for the harmonic
oscillator $H_0 = -\frac{1}{2}\Delta +\frac{|x|^2}{2}$ (see
e.g. \cite{LandauQ}):  
the pure point spectrum is 
\begin{equation*}
    \si_p(H_0)=\left\{ \frac{d}{2}+k=:\lambda_k \ ;\ k\in \N\right\},
  \end{equation*}
and the associated eigenfunctions are given by (tensor products of)
Hermite functions (denoted by $\psi_{jk}$, associated to $\lambda_k$),
which form a basis of $L^2(\R^d)$. So if  
\begin{equation*}
  u(0,x)= \sum_{j,k\in \N}\alpha_{jk}\psi_{jk}(x), 
\end{equation*}
then for all $t\in \R$,
\begin{equation*}
   u(t,x)= \sum_{j,k\in \N}\alpha_{jk}\psi_{jk}(x)e^{-i\(\frac{d}{2}+k\) t} . 
\end{equation*}
Perturbations of \eqref{eq:LSP} rather lead to KAM type results, in
the linear or in the nonlinear setting (see
e.g. \cite{GrTh11,Ku93}). In this paper, we study the large time behavior
of solutions to a nonlinear perturbation of a mixture of the above two
linear dynamics:
\begin{equation}
  \label{eq:NLSP}
  i\d_t u = Hu +\lambda |u|^{2\si}u, \quad (x,y)\in \R^n\times
  \R^{d-n},\quad H= -\frac{1}{2}\Delta_x
  +\frac{|x|^2}{2}-\frac{1}{2}\Delta_y, 
\end{equation}
with $d\ge 2$, $1\le n\le d-1$, $\lambda\in \C$ and $\si>0$. We show that if
the power of the 
nonlinearity is sufficiently large, then a scattering theory is
available: if $f(z)=|z|^{2\si}z$ is short range in space dimension
$d-n$, then the nonlinearity is negligible for large time. In the rest
of the introduction, we make this statement more precise. 
\smallbreak

We emphasize also the fact that the harmonic potential in
\eqref{eq:NLSP} corresponds to the standard modelling for magnetic
traps in the contexts of Bose-Einstein condensation (see
e.g. \cite{JosserandPomeau,PiSt}). So for instance, our result shows
that turning off the confinement in some of the directions may
suffice for the condensate to evolve asymptotically freely, and that
turning it off in all directions is not necessary: it is possible to
keep some properties of the linear confinement, associated to the
linear Hamiltonian $H$. 
 
\subsection{A quick review on scattering for NLS}
\label{sec:review}

Consider
\begin{equation}
  \label{eq:NLS}
   i\d_t u +\frac{1}{2}\Delta u= \lambda |u|^{2\si}u, \quad x\in \R^d,
\end{equation}
and introduce the space 
\begin{equation*}
  \Sigma = \left\{ f\in L^2(\R^d)\ ;\ \| f \|_{\Sigma}:=
    \left\lVert x 
      f\right\rVert_{L^2(\R^d)}+ \left\lVert \nabla 
      f\right\rVert_{L^2(\R^d)}<\infty\right\}. 
\end{equation*}
The following statement summarizes several results which can be found in
e.g. \cite{CazCourant,Ginibre}. 
\begin{theorem}\label{theo:NLS}
  Let $d\ge 1$, $\lambda\in \C$, $\si>0$ with $\si<2/(d-2)$ if $d\ge 3$. \\
$1.$ \emph{Existence of wave operators in $H^1$}. Suppose that $\si\ge
2/d$. Let $u_-\in
H^1(\R^d)$. There exist
$T>0$, depending on $|\lambda|$, $d$, $\sigma$ and $\|u_-\|_{H^1}$,
and a unique solution $u\in C((-\infty,-T];H^1)$ to \eqref{eq:NLS}  
such that
\begin{equation*}
  \left\| u(t) - e^{i\frac{t}{2}\Delta} u_-\right\|_{H^1}=
\left\|e^{-i\frac{t}{2}\Delta} u(t) - u_-\right\|_{H^1}\Tend t {-\infty}0.
\end{equation*}
$2.$ \emph{Existence of wave operators in $\Sigma$}. Let $u_-\in
\Sigma$. There exist
$T>0$, depending on $|\lambda|$, $d$, $\sigma$ and $\|u_-\|_\Sigma$,
and a unique solution $u\in C((-\infty,-T];\Sigma)$ to \eqref{eq:NLS}  
such that
\begin{equation*}
  \left\|e^{-i\frac{t}{2}\Delta} u(t) - u_-\right\|_{\Sigma}\Tend t {-\infty}0,
\end{equation*}
under the following assumption on $\si$:
\begin{equation*}
  \si>1 \text{ if } d=1,\quad \si>\frac{2}{d+2} \text{ if } d\ge 2. 
\end{equation*}
$3.$ \emph{Asymptotic completeness in $H^1$}. Suppose that $\lambda\in
\R$, with $\lambda \ge 0$, and $\si>2/d$. For all $u_0\in H^1(\R^d)$,
there exist a unique $u\in C(\R;H^1)$ solution to \eqref{eq:NLS} with
$u_{\mid t=0}=u_0$, and a unique $u_+\in H^1(\R^d)$ such that
\begin{equation*}
  \left\| u(t) - e^{i\frac{t}{2}\Delta} u_+\right\|_{H^1}=
\left\|e^{-i\frac{t}{2}\Delta} u(t) - u_+\right\|_{H^1}\Tend t {+\infty}0.
\end{equation*}
$4.$ \emph{Asymptotic completeness in $\Sigma$}. 
Suppose that $\lambda \ge 0$, and
$\si\ge \frac{2-d+\sqrt{d^2+12d+4}}{4d}$. For all $u_0\in \Sigma$, 
there exist a unique $u\in C(\R;\Sigma)$ solution to \eqref{eq:NLS} with
$u_{\mid t=0}=u_0$, and a unique $u_+\in \Sigma$ such that
\begin{equation*}
  \left\|e^{-i\frac{t}{2}\Delta} u(t) - u_+\right\|_{\Sigma}\Tend t {+\infty}0.
\end{equation*}
$5.$ \emph{Asymptotic completeness in $\Sigma$ for small data}. 
Let $\lambda \in \R$. Suppose that 
\begin{equation*}
  \si>1 \text{ if } d=1,\quad \si>\frac{2}{d+2} \text{ if } d\ge 2, 
\end{equation*}
and $\lambda \ge 0$ if $\si\ge 2/d$.  
If $\|u_0\|_\Sigma$ is sufficiently small,
there exist a unique $u\in C(\R;\Sigma)$ solution to \eqref{eq:NLS} with
$u_{\mid t=0}=u_0$, and a unique $u_+\in \Sigma$ such that
\begin{equation*}
  \left\|e^{-i\frac{t}{2}\Delta} u(t) - u_+\right\|_{\Sigma}\Tend t {+\infty}0.
\end{equation*}
$6.$ \emph{Existence of long range effects.} Suppose $\si\le 1/d$. If
$u\in C([T,\infty);L^2(\R^d)$ solution to \eqref{eq:NLS} and $u_+\in
L^2(\R^d)$ satisfy
\begin{equation*}
  \left\| u(t) - e^{i\frac{t}{2}\Delta} u_+\right\|_{L^2}=
\left\|e^{-i\frac{t}{2}\Delta} u(t) - u_+\right\|_{L^2}\Tend t {+\infty}0,
\end{equation*}
then necessarily, $u=u_+=0$. 
\end{theorem}
\begin{remark}
  Uniqueness is actually granted in smaller spaces than in the above
  statement, involving a mixed time-space norm which we have omitted
  to simplify the presentation. 
\end{remark}
\begin{remark}
  As pointed out in the above statement, the free group
  $e^{i\frac{t}{2}\Delta}$ is unitary on $H^1(\R^d)$, while it is not
  on $\Sigma$. 
\end{remark}
The above (positive) results all rely on (global in time) Strichartz
estimates. It is essentially the only tool to prove the first point,
while the second point uses the operator $x+it\nabla$, which provides
an explicit decay rate in time, 
to improve the assumption on $\si$. The third point relies on Morawetz
estimates (see also \cite{GiVe10,PlVe09}). The fourth point relaxes the
assumption of the third point on $\sigma$, since
\begin{equation*}
  \frac{2}{d}> \frac{2-d+\sqrt{d^2+12d+4}}{4d}>\frac{1}{d}. 
\end{equation*}
It relies on the pseudo-conformal evolution law, which provides a
rather simple relation concerning the time evolution of the $L^2$ norm
of $(x+it\nabla)u$. The fifth point is essentially the same as the
second one, plus a bootstrap argument. 
The final point means that if $\si\le 1/d$, linear
and nonlinear dynamics can no longer be (easily) compared, due to long
range effects. 

To be complete, we recall the above mentioned pseudo-conformal
evolution law: if $u$ solves \eqref{eq:NLS}, then 
\begin{align*}
  \frac{d}{dt}\( \frac{1}{2}\|(x+it\nabla_x)u\|_{L^2}^2 +
  \frac{\l t^2}{\si+1}\|u\|_{L^{2\si+2}}^{2\si+2}\)= \frac{\l
    t}{\si+1}(2-d\si) \|u\|_{L^{2\si+2}}^{2\si+2}. 
\end{align*}
This law was discovered in \cite{GV79Scatt}. A generalization to the
case of an external potential can be found in \cite{CazCourant}. Note
however that in the presence of an external potential, it seems that
the corresponding evolution law can be exploited only in the case
where the potential is \emph{exactly quadratic and isotropic}
(\cite{CaCCM}). In particular, in the case of a partial confinement as
in \eqref{eq:NLSP}, the law presented in \cite{CazCourant} seems
helpless as far as the description of large time behavior is
concerned. 
\subsection{Heuristics}
\label{sec:heu
ristics}

Leaving out all the technical aspects, scattering for \eqref{eq:NLS}
stems from the following remark. As $t\to +\infty$, we have
\begin{equation*}
  e^{i\frac{t}{2}\Delta }f(x)= \frac{1}{(it)^{d/2}} \widehat
    f\(\frac{x}{t}\)e^{i\frac{|x|^2}{2t}} +o(1)\text{ in }L^2(\R^d),
\end{equation*}
where the Fourier transform is normalized as 
\begin{equation*}
  {\mathcal F} f(\xi)=\widehat
  f(\xi)=\frac{1}{(2\pi)^{d/2}}\int_{\R^d}f(x)e^{-ix\cdot \xi} dx.
\end{equation*}
Assuming that the solution to \eqref{eq:NLS} behaves as the free
evolution of a scattering state $u_+$ as $t\to +\infty$, the nonlinear
potential in \eqref{eq:NLS} satisfies
\begin{equation*}
  |u(t,x)|^{2\si} \approx \frac{1}{t^{d\si}} \left|\widehat
    u_+\(\frac{x}{t}\)\right|^{2\si}.
\end{equation*}
The function on the right hand side belongs to $L^1_tL^\infty_x$
provided that ($u_+$ is sufficiently localized and) $\si>1/d$. From
the linear scattering point of view \cite{DG}, this is consistent
with the fact that $u$ behaves asymptotically like a solution to the
free equation \eqref{eq:LS}. In the case $\si\le 1/d$, the (nonlinear)
potential remains relevant at leading order for all time, no matter
how large, as proved in \cite{Barab}; long range effects were
described for the first time in \cite{Ozawa91} in the case of
\eqref{eq:NLS}. Note however that except in the cases $d=1$ and $d=2$,
Theorem~\ref{theo:NLS} introduces restrictions on the lower bound for
$\si$ which are not intuitive, and are likely to be ``only''
technical. 
\smallbreak

In the case of \eqref{eq:NLSP}, the $x$-part is not expected to yield
large time dispersion, while the $y$-part should provide large time
dispersion. From this perspective, it seems natural to expect a
scattering theory for \eqref{eq:NLSP} with the same numerology as in
Theorem~\ref{theo:NLS}, with $d$ replaced by $d-n$ as far as lower
bounds on $\si$ are concerned. 
\smallbreak

Such a problem is to be compared with the approach in  
\cite{TzVi12}. The authors study \eqref{eq:NLSP} where,
instead of considering the harmonic oscillator on $\R^n$, they
examine the Laplacian on a compact manifold without boundary. Also,
in \cite{HaPa13}, Z.~Hani and B.~Pausader consider a similar problem
in the particular case 
where $\lambda=1$, $\si=2$ (quintic, defocusing nonlinearity), $x\in
\T^2$ ($n=2$) and $d=3$. This problem is in addition both $L^2$ and
energy-critical, an aspect which introduces new difficulties compared
to \eqref{eq:NLSP} as it is treated in the present paper. In both
\cite{HaPa13} and \cite{TzVi12},
the nonlinearity in supposed to be short range in the $y$ variable so
that scattering in $H^1$ can be expected,
$\si\ge 2/(d-n)$, and small data scattering is established: for small
initial data (in spaces that we do not describe here), the solution to
the nonlinear equation is global in time, and behaves for large time
like a solution to the linear equation. 

More recently, in \cite{PTV-p}, the authors have considered the issue
of asymptotic completeness for the nonlinear Schr\"odinger equation
posed on $\R^N\times \T$, a case which corresponds to $n=1$ for
\eqref{eq:NLSP}. They prove that a scattering theory is available with
the same numerology as on $\R^N$, provided that the nonlinearity in
energy-subcritical in total space dimension $N+1$. Such a result is
qualitatively very  similar to our Theorem~\ref{theo:ac}. 
\smallbreak

Note that like with $H$ in
\eqref{eq:NLSP}, the dynamics in the $x$ variable prevents the
free dynamics from being as trivial as in the case of \eqref{eq:LS}. 
We will return to a comparison with these papers on a more
technical level in Sections~\ref{sec:product} and \ref{sec:strichartz}. 

\subsection{Main results}
\label{sec:main}
Note that the definition of $\Sigma$ has to be adapted to the present
notations:
\begin{equation*}
  \Sigma = \left\{ f\in L^2(\R^d)\ ;\ \| f \|_{\Sigma}:=
    \left\lVert x 
      f\right\rVert_{L^2(\R^d)}+ \left\lVert y 
      f\right\rVert_{L^2(\R^d)}+ \left\lVert \nabla_{x,y} 
      f\right\rVert_{L^2(\R^d)}<\infty\right\}. 
\end{equation*}
In view of
 \cite[Theorem~1.8]{Ca11}, we have: 
\begin{proposition}\label{prop:Cauchy}
   Let $d\ge 2$, $1\le n\le d-1$, $\lambda \in \R$, and $\si>0$ with
  $\si<2/(d-2)$ if $d\ge 3$. If $\si\ge 2/d$, assume in addition that
  $\lambda\ge 0$. For all $u_0\in \Sigma$, \eqref{eq:NLSP}
  has a unique solution 
  \begin{equation*}
    u\in C(\R;\Sigma)\cap L^{(4\si+4)/(d\si)}_{\rm
      loc}\(\R;W^{1,2\si+2}(\R^d)\) 
  \end{equation*}
such that $u_{\mid
    t=0}=u_0$. The following conservations hold:
  \begin{equation*}
    \frac{d}{dt}\|u(t)\|_{L^2(\R^d)}^2=\frac{d}{dt}\(\<u(t),Hu(t)\>
+\frac{\lambda}{\si+1}\|u(t)\|_{L^{2\si+2}(\R^d)}^{2\si+2} \)=0.
  \end{equation*}
\end{proposition}
Our first result regarding scattering theory is the existence of wave
operators, along with asymptotic completeness for small data:
\begin{theorem}\label{theo:wave}
  Let $d\ge 2$, $1\le n\le d-1$, $\lambda\in \C$, and $\si>0$ with
  $\si<2/(d-2)$ if $d\ge 3$. Suppose in addition that
  \begin{equation}
    \label{eq:sigmawave}
    \si>\frac{2d}{d+2}\frac{1}{d-n}.
  \end{equation}
$\bullet$ \emph{Existence of wave operators}. Let $u_-\in
\Sigma$. There exist
$T>0$, depending on $|\lambda|$, $d$, $n$, $\sigma$ and $\|u_-\|_\Sigma$,
and a unique solution $u\in C((-\infty,-T];\Sigma)$ to \eqref{eq:NLSP}  
such that
\begin{equation*}
  \left\|e^{it H} u(t) - u_-\right\|_{\Sigma}\Tend t {-\infty}0.
\end{equation*}
$\bullet$ \emph{Asymptotic completeness for small data}. Let $u_0\in
\Sigma$ and $u$ 
be the solution provided by Proposition~\ref{prop:Cauchy}. If
$\|u_0\|_\Sigma$ is sufficiently small, there exists a unique
$u_+\in \Sigma$ such that
\begin{equation*}
   \left\|e^{it H} u(t) - u_+\right\|_{\Sigma}\Tend t {+\infty}0. 
 \end{equation*}
\end{theorem}
\begin{remark}[Uniqueness]
  Uniqueness is stated rather loosely in the first point, for an extra
  property will be needed, which is too involved to be made precise
  here. See Section~\ref{sec:wave} for full details. 
\end{remark}
\begin{remark}[Lower bound on $\si$]
  The assumption \eqref{eq:sigmawave} is the expected one, except
  possibly for
  the factor $2d/(d+2)$. If $d=2$ (hence $n=1$), \eqref{eq:sigmawave}
  corresponds to the heuristics in the previous subsection. On the
  other hand if $d\ge 3$, we have $2d/(d+2)>1$, and the expected lower
  bound for $\si$ is not reached. Note that this factor is the same as
  in the second point in Theorem~\ref{theo:NLS}. It actually appears
  for the same technical reason as in the context of
  Theorem~\ref{theo:NLS}, as will be clear in Section~\ref{sec:wave}. 
\end{remark}
\begin{remark}[Values allowed for $n$]
  The assumptions of Theorem~\ref{theo:wave} always allow the value
  $n=1$ (one direction of confinement). The value $n=2$ is allowed for
  $d\ge 3$, since the range
  \begin{equation*}
    \frac{2d}{d+2}\frac{1}{d-2}<\si<\frac{2}{d-2}
  \end{equation*}
is non-empty. The value $n=3$ is allowed only if $d\ge 7$, and no
value $n\ge 4$ can be considered, since for $d\ge 5$,
\begin{equation*}
  \frac{2d}{d+2}\frac{1}{d-4}>\frac{2}{d-2}.
\end{equation*}
So in the physically relevant cases $d=2$ or $3$, any $n$ such that $1\le
n\le d-1$ is allowed. 
\end{remark}

Regarding asymptotic completeness for large data, the first remark is
that due to the 
anisotropy of the operator $H$, no analogue of the pseudo-conformal
evolution law, which is used in the proof of the fourth point in
Theorem~\ref{theo:NLS}, must be expected. 
Therefore, it is more
sensible to rely on Morawetz type
estimates, which we prove in a more general context than the framework
of \eqref{eq:NLSP} with a quadratic potential.

\begin{proposition}[Anisotropic Morawetz estimates]\label{prop:morawetz}
 Consider $H=-\frac12\Delta_x+V(x)-\frac12\Delta_y$, where $V$ is a
 real-valued potential depending only on $x\in\R^n$. Let  
$\lambda>0$ and assume that the solution $u$ to \eqref{eq:NLSP} with
$u_{\mid t=0}=u_0\in L^2(\R^d)$ exists globally, and satisfies
\begin{equation*}
\sup_{t\in\R}\|\nabla_yu(t)\|_{L^2(\R^d)}<\infty.
\end{equation*}
Then we have
\begin{equation}\label{eq:mor_mar}
\left\lVert |\nabla_y|^{\frac{3-(d-n)}{2}}R\right\rVert_{L^2_{t,
    y}(\R\times\R^{d-n})}^2\leq C\|u_0\|_{L^2(\R^d)}^3\sup_{t\in \R}
\|\nabla_yu(t)\|_{L^2(\R^d)},
\end{equation}
where 
\begin{equation*}
R(t, y):=\int_{\R^n}|u(t, x, y)|^2\,dx,
\end{equation*}
is the marginal of the mass density. 
\end{proposition}
Using these estimates, we can prove asymptotic completeness
results. Note however that since we consider energy-subcritical
problems on $\R^d$, we assume
\begin{equation*}
  \si<\frac{2}{(d-2)_+}.
\end{equation*}
On the other hand, proving asymptotic completeness like in $\R^{d-n}$
via Morawetz estimates 
($d-n$ is the ``scattering dimension'') is sensible, in view of the
third point of Theorem~\ref{theo:NLS}, provided that
\begin{equation*}
  \si>\frac{2}{d-n}.
\end{equation*}
The above two conditions are compatible if and only if $n=1$. This
restriction is expected to be only technical, since scattering should
occur under the weaker assumption
\begin{equation*}
  \si>\frac{1}{d-n},
\end{equation*}
but we lack of tools to decrease the value of $\si$ down to this
threshold. Note that in the small data case, asymptotic completeness
is established for $n=2$ provided that $d\ge 3$, for instance, from
Theorem~\ref{theo:wave}. 
On the other hand, it is our choice to state the result
only for $d\le 4$, since the case $d\ge 5$ would bring new
technicalities in the presentation, but no new real difficulty (see
also \cite{PTV-p}). 
\begin{theorem}[Asymptotic completeness]\label{theo:ac}
  Let $n=1$, $2\le d\le 4$,  and $\lambda >0$. Assume moreover:
  \begin{itemize}
  \item If $d=2$: $\si>2$.
  \item If $d=3$: $1<\si<2$.
  \item If $d=4$: $2/3<\si<1$. 
  \end{itemize}
For all $u_0\in \Sigma$, there exists a
  unique $u_+\in \Sigma$ such that the solution $u$ provided by
  Proposition~\ref{prop:Cauchy} satisfies
\begin{equation*}
  \left\|e^{it H} u(t) - u_+\right\|_{\Sigma}\Tend t {+\infty}0.
\end{equation*}
\end{theorem}
  
\subsection{Structure of the paper}
\label{sec:plan}

In Section~\ref{sec:product}, we recall some of the aspects in
\cite{HaPa13,TzVi12}, with emphasis on aspects related to Strichartz
estimates. In Section~\ref{sec:strichartz}, we propose a slight
generalization of a result from \cite{TzVi12}, and establish global in
time Strichartz estimates associated to $e^{-itH}$ which are isotropic
in space, in the same fashion as in \cite{HaPa13}, and as opposed to
\cite{PTV-p,TzVi12}. These estimates allow us to prove
Theorem~\ref{theo:wave} in Section~\ref{sec:wave}, thanks to suitable
vector-fields. Proposition~\ref{prop:morawetz} is established in
Section~\ref{sec:morawetz}, and Theorem~\ref{theo:ac}, in
Section~\ref{sec:ac}. Finally in an appendix, we present formal
arguments suggesting the existence of long range effects if $\si$ is
not sufficiently large.

\section{Nonlinear Schr\"odinger equation on product spaces}
\label{sec:product}
 
In this section, we recall some results which are highly related to
the framework of the present paper, at least qualitatively.
\smallbreak

A key step both in \cite{TzVi12} and in \cite{HaPa13} consists in
proving global in time Strichartz estimates, even though the free dynamics is
not fully dispersive due to the boundedness of the compact manifold
($\T^2$ in the case of \cite{HaPa13}). For $M^n$ an $n$-dimensional
compact manifold without boundary, consider
\begin{equation}
  \label{eq:LSM}
  i\d_t u+\frac{1}{2}\Delta_{M^n\times \R^{d-n}}u =F. 
\end{equation}
We introduce once and for all the notion of admissible pairs inspired
by Strichartz estimates on $\R^k$. The definition includes a notion of
``dispersive dimension'', and will be of constant use in this paper. 
\begin{definition}\label{def:adm}
  Let $k$ be a positive integer. A pair $(q,r)$ is $k$-admissible if $2\le r
  <\frac{2k}{k-2}$ ($2\le r\le\infty$ if $k=1$, $2\le r<
  \infty$ if $k=2$)  
  and 
$$\frac{2}{q}= k\left( \frac{1}{2}-\frac{1}{r}\right).$$
\end{definition}
\begin{proposition}[From \cite{TzVi12}]\label{prop:striTV}
  Let $d\ge 2$, $1\le n\le d-1$ and $M^n$ be an $n$-dimensional
  compact manifold without boundary. For $(x,y)\in M^n\times
  \R^{d-n}$, the following estimate holds: there exists $C_{r_1,r_2}$
  such that
  \begin{align*}
    \|e^{it\Delta_{x,y}}f\|_{L^{p_1}_t L^{r_1}_y L^2_x} +& \left\|\int_0^t
      e^{i(t-s)\Delta_{x,y}} F(s,x,y)ds\right\|_{L^{p_1}_t L^{r_1}_y
      L^2_x}\\
& \le C_{r_1,r_2}\(\|f\|_{L^2_{x,y}}+ \|F\|_{L^{p_2'}_t L^{r_2'}_y
      L^2_x}\),  
  \end{align*}
where the pairs $(p_1,r_1)$ and $(p_2,r_2)$ are $(d-n)$-admissible. 
\end{proposition}
The proof in \cite{TzVi12} relies on the fact that $\Delta_{M^n}$
possesses an eigenbasis in $L^2(M^n)$: by decomposing any solution to
\eqref{eq:LSM} on this eigenbasis, each coefficient solves a
Schr\"odinger equation on $\R^{d-n}$, and satisfies global Strichartz
estimates. Proposition~\ref{prop:striTV} follows by summing these
inequalities in $L^2(M^n)$ and invoking Minkowski inequality. In
Section~\ref{sec:generalization}, we present a generalization of this
result, which does not require any spectral analysis.

In the case where $d=3$, $n=2$ and $M^2=\T^2$, another family of estimates
has been established in \cite{HaPa13}. For $\gamma \in \Z$, set
$I_\gamma =2\pi[\gamma,\gamma+1)$.  
\begin{proposition}[From \cite{HaPa13}]\label{prop:striHP}
  Let $N\ge 1$ be dyadic, then 
  \begin{equation*}
    \left\|e^{it\Delta_{\R\times\T^2}}P_{\le
        N}u_0\right\|_{\ell^q_\gamma
      L^p_{t,x,y}(I_\gamma\times\R\times\T^2)}\lesssim
    N^{\(\frac{3}{2}-\frac{5}{p}\)} \|u_0\|_{L^2(\R\times\T^2)},
  \end{equation*}
whenever 
\begin{equation*}
  p>4\quad \text{and}\quad \frac{2}{q}+\frac{1}{p}=\frac{1}{2},
\end{equation*}
and where $P_{\le N}$ stands for a frequency cut-off. 
\end{proposition}
An important difference with the approach in \cite{TzVi12} is that in
the case $M^n=\T^n$, the eigenbasis of the Laplacian possesses a group
structure, which makes it possible to proceed with more explicit
computations. Extensive use of Strichartz estimates on $\T^2$
\cite{Bo93} (see also \cite{GiBBK})
is also made to prove Proposition~\ref{prop:striHP}. We emphasize
however that the proof of Proposition~\ref{prop:striHP} is rather
involved technically. We prove an analogous result in the framework of
\eqref{eq:NLSP} in Section~\ref{sec:harmo-global}. Even though the group
structure of the eigenfunctions of $H$ is lost (the eigenfunctions are
given by Hermite functions), computations
are less involved, and one does not have to face a loss of regularity
for local Strichartz estimates associated to the harmonic
oscillator.

\section{Global in time Strichartz estimates} 
\label{sec:strichartz}

\subsection{A slight generalization}\label{sec:generalization}

In this subsection, we generalize Proposition~\ref{prop:striTV} in a
case where the existence of an eigenbasis for the operator in the $x$
variable is not assumed. 
\begin{proposition}\label{prop:generalization}
  Let $d\ge 2$, $1\le n\le d-1$, $X^n$ be
an $n$-dimensional space, and $y\in \R^{d-n}$. Let $P$ be an operator
acting on the $x$ variable, but not on the $y$ variable. In particular, 
\begin{equation}\label{eq:commutes}
  [P,\Delta_y]=0. 
\end{equation}
Assume that the flow generated by $P$ is uniformly bounded on
$L^2(X^n)$, at least in the future,
\begin{equation*}
  \exists C>0,\quad \|e^{itP}v_0\|_{L^2(X^n)}\le C
  \|v_0\|_{L^2(X^n)},\quad \forall t\ge 0.
\end{equation*}
Then for all $u_0\in L^2(\R^{d})$, all $(d-n)$-admissible pairs
$(p_1,r_1)$ and $(p_2,r_2)$, there exists $C_{r_1,r_2}$ such that the
solution to 
\begin{equation*}
   i\d_t u + Pu+\Delta_y u  =F, \quad u_{\mid t=0}=u_0,
\end{equation*}
satisfies:
  \begin{align*}
    \|u\|_{L^{p_1}_t(\R_+; L^{r_1}_y L^2_x)}  \le
    C_{r_1,r_2}\(\|u_0\|_{L^2_{x,y}}+ \|F\|_{L^{p_2'}_t (\R_+;L^{r_2'}_y 
      L^2_x)}\). 
  \end{align*}
\end{proposition}
For instance, $X^n$ may be any manifold without boundary, and
$P$ a Schr\"odinger operator, including  a real-valued
external potential or magnetic field, with decay or sign assumptions
to ensure that the flow is well-defined on $L^2(X^n)$. In this cases,
it may happen that no spectral theory  like in \cite{TzVi12} is
available. In general, the $X^n$ part can be viewed as a black box, and
Proposition~\ref{prop:generalization} is a way to take advantage of
the Schr\"odinger dispersion on $\R^{d-n}$. 
\begin{example}
  In
the case of 
Schr\"odinger operators on $\R^n$, $P=-\Delta_x+V(x)$, it suffices to
consider  $V= V_1+V_2$, with (see \cite[p.~199]{ReedSimon2}) $V_1$,
$V_2$ real-valued and measurable, and
\begin{itemize}
\item $V_1(x)\ge -a|x|^2-b$ for some constants $a$ and $b$.
\item $V_2\in L^p(\R^n)$ with $p\ge 2$ if $n\le 3$, $p>2$ if $n=4$,
  and $p\ge n/2$ if $n\ge 5$.
\end{itemize}
This includes in particular the case studied in this paper, $V(x)=|x|^2/2$.
\end{example}
\begin{example}
  Still on $\R^n$, $P$ may be any polynomial in $D_x=-i\d_x$, not
  necessarily elliptic, or more generally a real-valued Fourier
  multiplier. 
\end{example}

\begin{proof}
The commutation assumption
  \eqref{eq:commutes} implies
  \begin{equation*}
    e^{it(P+\Delta_y)} = e^{itP}e^{it\Delta_y} =
    e^{it\Delta_y}e^{itP}. 
  \end{equation*}
Using this remark, the assumption that $e^{itP}$ is bounded on
$L^2(X^n)$ for positive time, and the standard properties of the
Schr\"odinger group on $\R^{d-n}$, we have:
\begin{equation*}
  \|e^{it(P+\Delta_y)}\|_{L^2_{x,y}\to L^2_{x,y}}\lesssim 1,\quad \|
  e^{it(P+\Delta_y)} \|_{ L^1_yL_x^2 \to  L^\infty_y L^2_x}\lesssim
  \frac{1}{t^{(d-n)/2}},\quad t>0. 
\end{equation*}
Invoking \cite[Theorem~10.1]{KeTa98}, with 
\begin{equation*}
  B_0=H = L^2_{x,y},\quad B_1 = L^1_y L_x^2,
\end{equation*}
we have
\begin{align*}
  &\|e^{it(P+\Delta_y)}u_0\|_{L^p_t(\R_+; B_\theta^*)}\lesssim
  \|u_0\|_{H},\\
& \left\|\int_0^te^{i(t-s)(P+\Delta_y)}
  F(s)ds\right\|_{L^{p_1}_t(\R_+;B_{\theta_1}^*)}\lesssim
\|F\|_{L^{p_2'}_t(\R_+;B_{\theta_2})}, 
\end{align*}
where for $0\le \theta\le 1$, $B_\theta$ denotes the real interpolation space
$(B_0,B_1)_{\theta,2}$, and $2/p=(d-n)\theta/2$,
$p\ge 2$, $(p,\theta,d-n)\not = (2,1,1)$, 
hence the proposition.  
\end{proof}

\subsection{The case of partial harmonic
  confinement}\label{sec:harmo-global}

We split the Hamiltonian $H$ into two parts, the confining one and the
fully dispersive one, by denoting
\begin{equation*}
  H_1 = -\frac{1}{2}\Delta_x
  +\frac{|x|^2}{2};\quad H_2=-\frac{1}{2}\Delta_y.
\end{equation*}
As in the previous subsection, we note that both operators commute,
$ [H_1,H_2]=0,$
therefore their corresponding propagators also commute
\begin{equation*}
  e^{-itH}=e^{-it(H_1+H_2)}=e^{-itH_1}e^{-itH_2}= e^{-itH_2}e^{-itH_1}.
\end{equation*}
The classical formula 
\begin{equation*}
  e^{-itH_2}g(y) = \frac{1}{(2i\pi
    t)^{(d-n)/2}}\int_{\R^{d-n}}e^{i\frac{|y-y'|^2}{2t}}f(y')dy'
\end{equation*}
yields the standard global dispersive estimate
\begin{equation}\label{eq:dispfree}
  \|e^{-itH_2}\|_{L^1(\R^{d-n})\to L^\infty(\R^{d-n})}\le \frac{1}{\(2\pi
    |t|\)^{(d-n)/2}},\quad \forall t\not= 0. 
\end{equation}
Similarly, the fundamental solution associated to $H_1$ is given by
Mehler's formula \cite{Me66}: 
\begin{equation*}
  e^{-itH_1}f(x) = \frac{1}{(2i\pi \sin t)^{n/2}}\int_{\R^n}
  e^{\frac{i}{2\sin t}\((|x|^2+|x'|^2)\cos t-2x\cdot x' \)}f(x')dx',
\end{equation*}
where the square root of $\sin t$ means implicitly that the
singularities of the fundamental solution are taken into account:
every time $\sin t=0$, a phase shift appears (Maslov index; see
e.g. \cite{KRY,Zelditch83}). The only 
information that we shall actually use is that local in time
dispersive estimates are available, periodically in time:
\begin{equation}\label{eq:dispharmo}
  \|e^{-itH_1}\|_{L^1(\R^n)\to L^\infty(\R^n)}\le \frac{1}{\(2\pi
    |\sin t|\)^{n/2}},\quad \forall t\not\in \pi\Z. 
\end{equation}
A property that we will use crucially to establish global Strichartz
estimates for $e^{-itH}$ is the fact that the above right hand side is
periodic (see Remark~\ref{rem:irrat} below).  Accordingly, for
$\gamma\in \Z$, we set $I_\gamma = \pi[\gamma-1,\gamma+1)$. 
\begin{theorem}[Global Strichartz estimates]\label{theo:strichartz}
Let $d\ge 2$ and $1\le n\le d-1$. If $2\le r < \frac{2d}{d-2}$ and the
pairs $(q,r)$ and $(p,r)$ are, respectively, 
$d$-admissible and $(d-n)$-admissible, then the
following two inequalities hold, 
\begin{equation}\label{eq:strichartzhom}
    \left\|e^{-itH}u_0\right\|_{\ell^p_\gamma
      L^q(I_\gamma;L^r(\R^d))}\lesssim
     \|u_0\|_{L^2(\R^d)},
\end{equation}
and
\begin{equation}\label{eq:strichartzhomdual}
     \left\|\int_{t\in\R} e^{i t H}F(t)dt\right\|_{L^{2}(\R^d)}\lesssim \|F\|_{\ell^{p'}_\gamma
      L^{q'}(I_\gamma;L^{r'}(\R^d))}.
\end{equation}
Also, if $(p_1,q_1,r_1)$ and $(p_2,q_2,r_2)$ are two such triplets, then
\begin{equation}\label{eq:strichartzinhom}
\left\|\int_0^t e^{-i(t-s)H}F(s)ds\right\|_{\ell^{p_1}_\gamma
      L^{q_1}(I_\gamma;L^{r_1}(\R^d))}\lesssim \|F\|_{\ell^{p_2'}_\gamma
      L^{q_2'}(I_\gamma;L^{r_2'}(\R^d))}.
\end{equation}
\end{theorem}
The above statement can be understood as follows. Locally in time, we
have the same Strichartz estimates as for $e^{it\Delta}$, a fact that
is standard in the presence of a smooth potential growing at most
quadratically in space (see e.g. \cite{CazCourant}). In the present
case, these estimates are made global in time thanks to the dispersion
in the $y$ directions, corresponding to the free Hamiltonian $H_2$:
there are $d-n$ such directions, so the assumption for $(p,r)$ to be
$(d-n)$-admissible appears very natural. 
\begin{proof}
The scheme of this proof follows the same general method as for
standard Strichartz inequalities, by using duality and the $TT^*$ argument of Tomas-Stein.
The added novelty here is strongly inspired by the
arguments in \cite{CuVi11,HaPa13}.

The homogeneous inequalities \eqref{eq:strichartzhom} and \eqref{eq:strichartzhomdual} are dual to
each other and therefore equivalent.

We now prove \eqref{eq:strichartzhomdual}. Let $F, G \in C^\infty(\R \times \R^d)$. Then, 
\eqref{eq:strichartzhomdual} is equivalent to the bilinear inequality
\begin{multline}\label{eq:bilinearStrichartz}
      \left|\<\int_{t\in\R} e^{i t H}F(t)dt,\int_{s\in\R} e^{i s H}G(s)ds\>_{L^2}
      \right| \\ \lesssim \|F\|_{\ell^{p'}_\gamma
      L^{q'}(I_\gamma;L^{r'}(\R^d))}\|G\|_{\ell^{p'}_\gamma
      L^{q'}(I_\gamma;L^{r'}(\R^d))},
\end{multline}
which, again by duality, is in turn equivalent to
\begin{equation}\label{eq:globalinhomstrichartz}
      \left\| \int_{s\in\R} e^{-i (t-s) H}G(s)ds \right\|_{\ell^{p}_\gamma
      L^{q}(I_\gamma;L^{r}(\R^d))} \lesssim \|G\|_{\ell^{p'}_\gamma
      L^{q'}(I_\gamma;L^{r'}(\R^d))}.
\end{equation}
To prove \eqref{eq:bilinearStrichartz} we use a partition of unity in time
$$ \sum_{\gamma \in \Z} \chi(t-\pi \gamma)=1, \quad \forall {t \in \R}
\qquad \mbox{with} \quad \mbox{supp}\chi \subset [-\pi, \pi].$$
We thus have
\begin{multline*}
\<\int_{t\in\R} e^{i t H}F(t)dt,\int_{s\in\R} e^{i s H}G(s)ds\>_{L^2}
 \\ = \sum_{\alpha, \beta \in \Z} \iint_{t,s \in \R}
 \< \chi(t-\pi \alpha) e^{i t H}F(t), \chi(t-\pi \beta) e^{i s H}G(s)\>_{L^2} dt\,ds.
\end{multline*}
This sum is now bounded by splitting it into the diagonal terms, $|\alpha-\beta|\le 1$, and
the non-diagonal ones, $|\alpha-\beta|> 1$.

The sum of the diagonal terms is bounded as follows
\begin{eqnarray}
\lefteqn{\sum_{|\alpha-\beta|\le 1} \left| \iint_{t,s \in \R}
\< \chi(t-\pi \alpha) e^{i t H}F(t), \chi(t-\pi \beta) e^{i s H}G(s)\>_{L^2} dt\,ds \right|}
\nonumber \\
& & \lesssim \sum_{|\alpha-\beta|\le 1} \left\| \int_{-\pi}^{\pi}
\chi(t) e^{i (t+\pi \alpha) H}F(t+\pi \alpha) dt \right\|_{L^2}
\left\| \int_{-\pi}^{\pi}
\chi(s) e^{i (s+\pi \beta) H}G(s+\pi \beta) ds \right\|_{L^2} \nonumber \\
& & = \sum_{|\alpha-\beta|\le 1} \left\| \int_{-\pi}^{\pi}
\chi(t) e^{i t H}F(t+\pi \alpha) dt \right\|_{L^2}
\left\| \int_{-\pi}^{\pi}
\chi(s) e^{i s H}G(s+\pi \beta) ds \right\|_{L^2} \nonumber,
\end{eqnarray}
where, in this last step, we have used the fact that the propagator $e^{-i t H}$ is unitary
in $L^2$. At this point, we use Strichartz estimates in their dual form (analogous to \eqref{eq:strichartzhomdual}), which are available locally in time, for the partially confined
operator $H$, with $d$-admissible pairs $(q,r)$,
\begin{multline}\label{eq:diagonal}
\lesssim \sum_{|\alpha-\beta|\le 1} \left\|F(\cdot+\pi \alpha) \right\|_{L^{q'}([-\pi,\pi];L^{r'}(\R^d))}
	\left\| G(\cdot+\pi \beta)  \right\|_{L^{q'}([-\pi,\pi];L^{r'}(\R^d))}  \\
	\lesssim \left\|F \right\|_{\ell^{2}_\gamma
		L^{q'}(I_\gamma;L^{r'}(\R^d))}
	\left\| G  \right\|_{\ell^{2}_\gamma
		L^{q'}(I_\gamma;L^{r'}(\R^d))}.
\end{multline}
For the non-diagonal part of the sum --- the crucial and deeper one
--- we proceed with a different approach, involving a localized $TT^*$ 
argument.
\begin{eqnarray}
\lefteqn{\sum_{|\alpha-\beta|> 1} \left| \iint_{t,s \in \R}
	\< \chi(t-\pi \alpha) e^{i t H}F(t), \chi(t-\pi \beta) e^{i s H}G(s)\>_{L^2} dt\,ds \right|}
\nonumber \\
& & = \sum_{|\alpha-\beta|> 1} \left| \iint_{t,s \in \R} \chi(t)\chi(s)
	\< F(t+\pi\alpha), e^{-i(t- s + \pi(\alpha-\beta)) H}G(s+\pi\beta)\>_{L^2} dt\,ds \right| \nonumber \\
	& & \lesssim \sum_{\substack{\beta \in \Z\\ |\gamma |>1}} 
	\left| \iint_{(t,s) \in [-\pi,\pi]^2} \chi(t)\chi(s)
	\< F(t+\pi(\beta+\gamma)), e^{-i(t- s + \pi\gamma) H}G(s+\pi\beta)\>_{L^2} dt\,ds \right| \nonumber \\
	& & \lesssim \sum_{\substack{\beta \in \Z\\ |\gamma |>1}} 
	\iint_{(t,s) \in [-\pi,\pi]^2} 
	\left\| F(t+\pi(\beta+\gamma))\right\|_{L^{r'}}
	\left\| e^{-i(t- s + \pi\gamma) H}G(s+\pi\beta) \right\|_{L^r}. \nonumber
\end{eqnarray}
We have now reached the stage where we use the
time decay estimate of the dispersive propagator. For that, we 
exploit the commutativity property of the partial propagators
$e^{-itH_1}$ and $e^{-itH_2}$, 
as well as the periodicity in time of $e^{-itH_1}$.

Thus, following the remarks at the beginning of this section, we know that
$$e^{-i(t- s + \pi\gamma) H}=e^{-i(t- s + \pi\gamma) H_1} e^{-i(t- s + \pi\gamma) H_2}.$$
The propagator for $H_2$ corresponds to the free Schr\"odinger case in the $y$ variables,
thus has a global time decay with rate $t^{-(d-n)/2}$, as seen in \eqref{eq:dispfree}. Whereas the propagator for $H_1$ corresponds to the time periodic harmonic oscillator in the $x$ variables, with local in time decay with rate
$t^{-n/2}$ from \eqref{eq:dispharmo}. So, doing the full $L^1(\R^d) \to L^\infty(\R^d)$ estimate 
for the propagator of $H$ by first splitting
the variables $x$ and $y$, then applying the partial propagators one at a time for each of the corresponding variables and using the periodicity of $H_1$, we obtain the combined time decay rate
$$\|e^{-i(t- s + \pi\gamma) H}f\|_{L^\infty_{x,y}(\R^d)}\lesssim
\frac{1}{|t-s|^{n/2}}\frac{1}{|\gamma|^{(d-n)/2}}\|f\|_{L^1_{x,y}(\R^d)},$$
which is local in the $t,s$ variables (from the partially confining operator $H_1$ acting on the $x$ variables) and global in $\gamma$ (from the free operator $H_2$ acting on the $y$ variables).
Finally, interpolating the above $L^1 \to L^\infty$ estimate with the $L^2$ conservation, we 
obtain
the intermediate $L^{r'} \to L^{r}$ decay to be used within the $TT^*$ proof of the Strichartz estimates,
\begin{equation*}
\|e^{-i(t- s + \pi\gamma) H}f\|_{L^r_{x,y}(\R^d)}\lesssim
\frac{1}{|t-s|^{n(1/2-1/r)}}
\frac{1}{|\gamma|^{(d-n)(1/2-1/r)}}\|f\|_{L^{r'}_{x,y}(\R^d)},
\end{equation*}
for $2\le r \le \infty$. 

So, plugging this estimate back into the non-diagonal sum inequality,
we thus obtain
\begin{align*}
&\lesssim \sum_{\substack{\beta \in \Z\\ |\gamma |>1}} 
	\iint_{(t,s) \in [-\pi,\pi]^2} 
	\left\| F(t+\pi(\beta+\gamma))\right\|_{L^{r'}}
	\frac{1}{|t-s|^{n(\frac{1}{2}-\frac{1}{r})}}
	\frac{1}{|\gamma|^{(d-n)(\frac{1}{2}-\frac{1}{r})}}\|G(s+\pi\beta)\|_{L^{r'}} \\
&  \lesssim \sum_{\substack{\beta \in \Z\\ |\gamma |>1}} 
\frac{1}{|\gamma|^{(d-n)(\frac{1}{2}-\frac{1}{r})}}
\iint_{(t,s) \in [-\pi,\pi]^2} 
\frac{1}{|t-s|^{n(\frac{1}{2}-\frac{1}{r})}}
\left\| F(t+\pi(\beta+\gamma))\right\|_{L^{r'}}
\|G(s+\pi\beta)\|_{L^{r'}}\\
& \lesssim \sum_{\substack{\beta \in \Z\\ |\gamma |>1}} 
\frac{1}{|\gamma|^{(d-n)(\frac{1}{2}-\frac{1}{r})}}
\|F(\cdot+\pi(\beta+\gamma))\|_{L^{q'}([-\pi,\pi];L^{r'}(\R^d))}
\|G(\cdot+\pi\beta)\|_{L^{q'}([-\pi,\pi];L^{r'}(\R^d))} ,
\end{align*}
where we have used the Hardy-Littlewood-Sobolev inequality in the
$t,s$ variables, in the last 
step, for $(q,r)$ now an $n$-admissible pair. To conclude this
estimate, we again 
apply the Hardy-Littlewood-Sobolev inequality, this time in its discrete
version for the $\gamma$ variable, to obtain 
\begin{equation}\label{eq:offdiagonal}
\lesssim \left\|F \right\|_{\ell^{p'}_\gamma
	L^{q'}(I_\gamma;L^{r'}(\R^d))}
\left\| G  \right\|_{\ell^{p'}_\gamma
	L^{q'}(I_\gamma;L^{r'}(\R^d))},
\end{equation}
whenever
$$\frac 2 p \le (d-n)\left(\frac 1 2 - \frac 1 r \right),$$
in particular for $(p,r)$ a $(d-n)$-admissible pair.

To finish the proof of \eqref{eq:bilinearStrichartz} we gather the
diagonal and non-diagonal estimates, \eqref{eq:diagonal} and
\eqref{eq:offdiagonal}, to obtain 
\begin{multline*}
\left|\<\int_{t\in\R} e^{i t H}F(t)dt,\int_{s\in\R} e^{i s H}G(s)ds\>_{L^2}
\right| \\ \lesssim
\left\|F \right\|_{\ell^{2}_\gamma
	L^{q_1'}(I_\gamma;L^{r'}(\R^d))}
\left\| G  \right\|_{\ell^{2}_\gamma
	L^{q_1'}(I_\gamma;L^{r'}(\R^d))}\\
+
\left\|F \right\|_{\ell^{p'}_\gamma
	L^{q_2'}(I_\gamma;L^{r'}(\R^d))}
\left\| G  \right\|_{\ell^{p'}_\gamma
	L^{q_2'}(I_\gamma;L^{r'}(\R^d))},
\end{multline*}
where the pair $(q_1,r)$ is $d$-admissible, while $(q_2,r)$ is $n$-admissible. Therefore
$q'_2 < q'_1$ and using the inclusion $L^{q'_1}(I_\gamma) \subset L^{q'_2}(I_\gamma)$, while
$\ell^{p'} \subset \ell^{2}$, we obtain \eqref{eq:bilinearStrichartz}. Thus, we have established
\eqref{eq:strichartzhom} and \eqref{eq:strichartzhomdual}, as well as \eqref{eq:globalinhomstrichartz}.

There still remains establishing \eqref{eq:strichartzinhom} to conclude the proof of the theorem. But this 
finite interval version follows from the similar global integral estimate \eqref{eq:globalinhomstrichartz} by using the now standard Christ-Kiselev lemma, together with interpolation between 
admissible triplets, just as in the analogous final stage of a classical Strichartz estimate proof.
\end{proof}

\begin{remark}[More general potentials]\label{rem:irrat}
The two properties of $e^{-it H_1}$ that we have used in the proof of
Theorem~\ref{theo:strichartz} are the existence of local dispersive
estimates, and the fact that such estimates are available periodically
in time. The first point remains for smooth potentials which grow at
most quadratically, from \cite{Fujiwara}. The fact that these
dispersive estimates are 
  available periodically in time remains for instance when the
  harmonic potential is suitably perturbed with a smooth one (Schwartz class),
  so the new operator has the same spectrum as the harmonic
  oscillator; this has been established in the case $n=1$ in
  \cite{McKeTr81}.  
Also, we may consider more general quadratic potentials, when $n\ge 2$, 
\begin{equation*}
  \sum_{j=1}^n \omega_j^2\frac{x_j^2}{2},\quad \omega_j>0.
\end{equation*}
If the $\omega_j$'s
  are rationally dependent, then the above proof can easily be
  repeated, since local in time dispersive estimates remain, as well
  as time periodicity of the propagator (each component leads to
  dispersive estimates which are $\pi/\omega_j$-periodic in time). 
On the other hand, if the $\omega_j$'s are rationally independent,
either by using Hermite functions (eigen decomposition) or by invoking
Mehler's formula, we see that the time periodicity is lost. 
  The argument that we have followed to obtain global Strichartz
  estimates no longer 
  works; the dynamics might be rather involved, typically due to the
  presence of small divisors. On the other hand, since
  Proposition~\ref{prop:generalization} only assumes that the flow
  associated to the operator $P$ ($=\frac{1}{2}\Delta_x +V(x)$ here),
  is bounded on $L^2$, its conclusion remains valid for \emph{all}
  quadratic potentials. 
\end{remark}

\section{Existence of wave operators and 
  small data scattering}
\label{sec:wave}

In this section, we prove Theorem~\ref{theo:wave}. The existence of
wave operators is obtained by a fixed
point argument in a suitable space, constructed in view of global
Strichartz estimates provided by
Theorem~\ref{theo:strichartz}. Asymptotic completeness for small data
follows essentially from the same estimates, plus a standard bootstrap
argument. 

\subsection{Vector fields}
\label{sec:Z}

As pointed out in the introduction, proving the existence of wave
operator in $\Sigma$ in the absence of external potential, that is for
\eqref{eq:NLS},  relies on the use of the
vector field $x+it\nabla$. In the case of \eqref{eq:NLSP}, this means
that we consider
\begin{equation*}
  J = y+it\nabla_y = it\; e^{i|y|^2/(2t)}\nabla_y \(e^{-i|y|^2/(2t)}\cdot\).
\end{equation*}
In view of the second identity, $J$ acts on the gauge invariant nonlinearity
$|u|^{2\si}u$ like $\nabla_y$, and Gagliardo-Nirenberg inequalities
involving $\nabla_y$ export to similar inequalities in terms of $J$,
with an extra time dependent factor:
\begin{equation}\label{eq:GN0}
  \|f\|_{L^r(\R^d)}\le
  \frac{C_r}{|t|^{(d-n)(1/2-1/r)}}\|f\|_{L^2(\R^d)}^{1-d(1/2-1/r)}
\|(\nabla_x,J)f\|_{L^2(\R^d)}^{d(1/2-1/r)}, 
\end{equation}
for all $2\le r\le \frac{2d}{d-2}$ ($2\le r<\infty$ if $d=2$). 
Finally, $J$ commutes with the linear part of \eqref{eq:NLSP},
$ \left[ i\d_t -H,J\right]=0$.

We now turn to the $x$-dependent part of $H$. In view of the above
weighted Gagliardo-Nirenberg inequality, we will get the same time
decay as on $\R^{d-n}$ provided that we can estimate $\nabla_x
u$ (in $L^2$). However, unlike $J$, $\nabla_x$ does not commute with
the linear 
part of \eqref{eq:NLSP}:
\begin{equation*}
  \left[ i\d_t -H,\nabla_x\right]= -x.
\end{equation*}
We remark that 
\begin{equation*}
   \left[ i\d_t -H,x\right]= \nabla_x,
\end{equation*}
so it is possible to obtain a closed system of estimates for
$(u,Ju,\nabla_x u,xu)$. This remark can be used to construct a
solution to \eqref{eq:NLSP} locally in time, thanks to the Gronwall
lemma; see \cite{Oh} (see also \cite{CaCCM}). However, the above
commutators generate an exponentially growing factor after the use of
the Gronwall lemma, an aspect which ruins the algebraic time decay provided
by \eqref{eq:GN0} when large time is considered, like in scattering
theory. Therefore, we need to replace the operators 
$\nabla_x$ and $x$ by adapted operators which enjoy properties similar
to those of $J$. In the case of the harmonic potential, such operators
are available: see e.g. \cite{CaCCM}. We introduce
\begin{align*}
  A_1(t) & = x\sin t -i \cos t\nabla_x,&&\quad A_2(t) =
  x\cos t +i \sin t\nabla_x \\
A_3(t) & = -i\nabla_y,&&\quad A_4(t) =
  y+it\nabla_y.
\end{align*}
To ease the forthcoming formulas, we also set
\begin{equation*}
  A_0 = {\rm Id}. 
\end{equation*}
The operators $A_3$ and $A_4$ are the standard ones in the absence of
external potential. The operators $A_1$ and $A_2$ are adapted to the
harmonic potential, and satisfy
\begin{equation*}
  \begin{pmatrix}
    A_1(t)\\
-A_2(t)
  \end{pmatrix}
= 
\begin{pmatrix}
  \sin t & \cos t\\
-\cos t & \sin t
\end{pmatrix}
\begin{pmatrix}
  x \\
-i\nabla_x
\end{pmatrix}.
\end{equation*}
Indeed, the operators $A_1$ and $A_2$ account for the
fact that the harmonic oscillator rotates the phase space at a speed
which is uniform in $(x,\xi)$. In particular, we have the pointwise
identity for, say, $f\in \Sch(\R^d)$,
\begin{equation}
  \label{eq:energy-rotation}
  |A_1(t)f(x,y)|^2+|A_2(t)f(x,y)|^2 = |x f(x,y)|^2 + |\nabla_x
  f(x,y)|^2,\quad \forall t\in \R. 
\end{equation}
\begin{lemma}\label{lem:Z}
 For $1\le j\le 4$, the operators $A_j$ satisfy the following properties.
 \begin{itemize}
 \item They correspond to the conjugation of gradient and momentum by
   the free flow,
   \begin{align*}
   &  A_1 (t) = e^{-itH}(-i\nabla_x)e^{itH},\quad A_2(t) =
     e^{-itH}x\, e^{itH} ,\\
&  A_3 (t) = e^{-itH}(-i\nabla_y) e^{itH},\quad A_4(t) =
     e^{-itH}y\, e^{itH}.
   \end{align*}
Therefore, they commute with the linear part of \eqref{eq:NLSP}:
   $[i\d_t-H,A_j(t)]=0$.
\item They act on gauge invariant nonlinearities like derivatives. In
  particular, we have the pointwise estimate
  \begin{equation*}
    \left| A_j(t)\(|u|^{2\si}u\) \right|\lesssim
    |u|^{2\si}|A_j(t)u|.
  \end{equation*}
\item Weighted Gagliardo-Nirenberg inequalities: for all  $2\le r\le
  \frac{2d}{d-2}$ ($2\le r<\infty$ if $d=2$), 
  there exists $K_r$ such that for all $f\in
  \Sigma$,
  $t\not =0$,
  \begin{equation*}
      \|f\|_{L^r(\R^d)}\le
  \frac{K_r}{\<t\>^{(d-n)(1/2-1/r)}}\|f\|_{L^2(\R^d)}^{1-\delta}
\(\sum_{j=1}^4\|A_j(t)f\|_{L^2(\R^d)}\)^{\delta}
  \end{equation*}
with
  $\delta = d(1/2-1/r)$.
 \end{itemize}
\end{lemma}
The last two points follow directly from the ``nonlinear'' counterpart
of the ``linear'' factorization of the first point:
\begin{align*}
  & A_1(t)f = -i\cos t\, e^{-\frac{i}{2}|x|^2 \tan t }
  \nabla _x\(e^{\frac{i}{2}|x|^2 \tan t}
  f \),\\
& A_2(t)f = i\sin t\, e^{i|x|^2 / (2\tan t) }
  \nabla _x\(e^{-i|x|^2 /(2 \tan t) }
  f \),\\
& A_4(t)f = it\,  e^{i|y|^2/(2t)}\nabla_y\( e^{-i|y|^2/(2t)}f\). 
\end{align*}
We emphasize that analogous operators are available for anisotropic
quadratic potentials. On the other hand, the existence of operators
enjoying the above three properties seems to be bound to potentials
which are polynomials of degree at most two (\cite{CaMi04}). 

In view of the first point in Lemma~\ref{lem:Z}, we have the
equivalence, in the sense of equivalence of norms:
\begin{align*}
  \|e^{itH} u(t)-u_-\|_\Sigma & \sim  
\sum_{j=0}^4\|A_j(t)u(t)-A_j(t)e^{-itH}  u_-\|_{L^2}\\
& \sim 
\sum_{j=0}^4\|A_j(t)u(t)-e^{-itH} A_j(0) u_-\|_{L^2}.
\end{align*}
\subsection{The fixed point argument}
\label{sec:fixed}
 Let $u_-\in \Sigma$. For $\gamma_0\le -1$, we denote 
 \begin{equation*}
   J_{\gamma_0} = (-\infty, \pi(\gamma_0+1)) = \bigcup_{\gamma\le
     \gamma_0}I_\gamma, 
 \end{equation*}
and we set
\begin{align*}
  X_{\gamma_0}& =\Big\{ u \in L^\infty(J_{\gamma_0};H^1(\R^d))\ ;\
  \text{ for all }j\in \{0,\dots,4\},\\
& \|A_j u\|_{L^\infty(J_{\gamma_0};L^2)} \le 2 \|A_j(0)u_-\|_{L^2},\quad
 \|A_j u\|_{\ell^p_{\gamma\le \gamma_0}L^q(I_\gamma;L^r)} \le 2C_r
\|A_j(0)u_-\|_{L^2}\Big\}, 
\end{align*}
where $C_r$  stems from Theorem~\ref{theo:strichartz}. The indices $p,q,r$ are
precised below: in view of Theorem~\ref{theo:strichartz}, they are
such that $(q,r)$ is $d$-admissible, and $(p,r)$ is $(d-n)$-admissible. 
Theorem~\ref{theo:wave} follows from a fixed point argument on
Duhamel's formula 
\begin{equation*}
  u(t) = e^{-itH}u_- -i\lambda\int_{-\infty}^t
  e^{-i(t-s)H}\(|u|^{2\si}u\)(s)ds=:\Phi(u)(t), 
\end{equation*}
applied in $X_{\gamma_0}$ for $\gamma_0\ll -1$ and $(p,q,r)$ well
chosen. Thanks to Theorem~\ref{theo:strichartz} and Lemma~\ref{lem:Z},
the argument becomes 
rather close to the proof the second point in Theorem~\ref{theo:NLS},
for which the reader is referred to, e.g., \cite{CazCourant,Ginibre}. 
\begin{proof}[Proof of Theorem~\ref{theo:wave}]
  The key step, for the stability of $X_{\gamma_0}$ under the action
  of the map $\Phi$ as well as for the contraction of $\Phi$ on
  $X_{\gamma_0}$ with respect to the weaker norm $\ell^p_{\gamma\le
    \gamma_0}L^q(I_\gamma;L^r)$ (Kato's method, see \cite{CazCourant}:
  $X_{\gamma_0}$ 
  equipped with that norm is complete, so it is not necessary to prove
  contraction for norms involving derivatives of $\Phi$), consists in
  controlling $B(|u|^{2\si}u)$ in, say, $\ell^{p'}_{\gamma\le
    \gamma_0}L^{q'}(I_\gamma;L^{r'})$, for all $B\in \{{\rm Id},
  A_1,\dots,A_4\}$. This is so thanks to Strichartz estimates, and
 to the first point in Lemma~\ref{lem:Z}. In view of the
  second point in Lemma~\ref{lem:Z} and H\"older inequality,
  \begin{equation*}
    \left\| B(|u|^{2\si}u)\right\|_{\ell^{p'}_{\gamma\le
    \gamma_0}L^{q'}(I_\gamma;L^{r'})}\lesssim \|u\|_{\ell^\theta_{\gamma\le
    \gamma_0}L^\omega(I_\gamma;L^k)}^{2\si}\|Bu\|_{\ell^p_{\gamma\le
    \gamma_0}L^q(I_\gamma;L^r)}, 
  \end{equation*}
\emph{with a constant independent of $\gamma_0$}, 
provided that
\begin{equation}\label{eq:holder}
  \frac{1}{p'}=\frac{1}{p}+\frac{2\si}{\theta},\quad
\frac{1}{q'}=\frac{1}{q}+\frac{2\si}{\omega},\quad
\frac{1}{r'}=\frac{1}{r}+\frac{2\si}{k}.
\end{equation}
If we can find $(p,q,r)$ with the above algebraic requirements, such
that 
\begin{equation*}
  u \in X_{\gamma_0}\Longrightarrow u\in \ell^\theta_{\gamma\le
    \gamma_0}L^\omega(I_\gamma;L^k),
\end{equation*}
with $(\theta,\omega,k)$ given like above, for some \emph{finite}
$\theta$, then the result follows, up to taking $|\gamma_0|$
sufficiently large. 
\smallbreak

Let $u\in X_{\gamma_0}$. In view of the last point in
Lemma~\ref{lem:Z},
\begin{equation*}
  \|u(t)\|_{L^k(\R^d)} \le \frac{C}{|t|^{(d-n)(1/2-1/k)}}, \quad t< \pi(\gamma_0+1),
\end{equation*}
where $C$ depends only on $d,n,k$ and $\|u_-\|_{\Sigma}$. Since
$|I_\gamma|=2\pi$, we infer for any $\omega$, $\gamma\le \gamma_0$,
\begin{equation*}
  \|u(t)\|_{L^\omega(I_\gamma;L^k(\R^d))} \le (2\pi)^{1/\omega}
  \|u(t)\|_{L^\infty(I_\gamma;L^k(\R^d))} \lesssim
  \frac{1}{|\gamma|^{(d-n)(1/2-1/k)}}.   
\end{equation*}
Therefore, 
\begin{equation*}
  X_{\gamma_0} \subset \ell^\theta_{\gamma\le
    \gamma_0}L^\omega(I_\gamma;L^k)\text{ as soon as
  }(d-n)\(\frac{1}{2}-\frac{1}{k}\) \theta>1. 
\end{equation*}
We distinguish two cases: $d=2$ and $d\ge 3$. 
\smallbreak

\noindent{\bf Case $d=2$.} Then necessarily, $n=1$.
We take $r = 2 +\eps$, with $0<\eps\ll 1$. In view of
\eqref{eq:holder}, we compute:
\begin{equation*}
  k=\frac{2\si(2+\eps)}{\eps},\quad p 
  \frac{4(2+\eps)}{\eps},\quad q=\frac{2(2+\eps)}{\eps},\quad \theta= 2\si
  \frac{4+2\eps}{4+\eps},\quad \omega= (2+\eps)\si,
\end{equation*}
hence
\begin{equation*}
 (d-n)\(\frac{1}{2}-\frac{1}{k}\) \theta = \frac{4\si
   +2(\si-1)\eps}{4+\eps}.  
\end{equation*}
This is larger than $1$ for $\eps>0$ sufficiently small as soon as
$\si>1$. 
\smallbreak

\noindent{\bf Case $d\ge 3$.} We consider the largest possible value
for $k$, which corresponds to the endpoint of Sobolev embedding,
\begin{equation*}
  k= \frac{2d}{d-2}. 
\end{equation*}
In view of \eqref{eq:holder}, we compute 
\begin{equation*}
  r = \frac{2d}{d-(d-2)\si}. 
\end{equation*}
Since the nonlinearity is energy-subcritical, $\si<2/(d-2)$, we have
$2\le r<2d/(d-2)$, so this value is acceptable. We also have
\begin{equation*}
  p = \frac{4d}{(d-n)\si(d-2)},\quad q = \frac{4}{\si(d-2)}, \quad
  \theta = \frac{4d \si}{2d-(d-n)(d-2)\si},
\end{equation*}
where the last formula defines a positive $\theta$, since
$\si<2/(d-2)$ and $n>0$. We infer
\begin{equation*}
 (d-n)\(\frac{1}{2}-\frac{1}{k}\) \theta =
 \frac{d-n}{d}\frac{4d \si}{2d-(d-n)(d-2)\si}.
\end{equation*}
This quantity is larger than $1$ if and only if
\begin{equation*}
  \si>\frac{2d}{d+2}\frac{1}{d-n},
\end{equation*}
which is the condition stated in Theorem~\ref{theo:wave}. 
\end{proof}
\subsection{Asymptotic completeness for small data}
\label{sec:ac-small}

We now turn to the proof of the second part of
Theorem~\ref{theo:wave}. 
 Resume exactly the same computations as the previous subsection. For
 $B\in \{{\rm Id}, 
  A_1,\dots,A_4\}$,
  \begin{equation*}
    \|Bu\|_{\ell^p_{\gamma\ge 1}L^q(I_\gamma;L^r)}\lesssim
    \|B(0)u_0\|_{L^2} +
    \|u\|_{\ell^\theta_{\gamma\ge 1}L^\omega(I_\gamma;L^k)}^{2\si}
\|Bu\|_{\ell^p_{\gamma\ge 1}L^q(I_\gamma;L^r)} .
  \end{equation*}
For $t\ge 0$, we estimate,
in view of the last point in
Lemma~\ref{lem:Z},
\begin{equation*}
  \|u(t)\|_{L^k(\R^d)} \le
  \frac{C}{\<t\>^{(d-n)(1/2-1/k)}}\|u_0\|^{1-\delta}\(\sum_{j=1}^4
  \|A_j(t)u\|_{L^2}\)^{\delta},
\end{equation*}
with $\delta=d(1/2-1/k)$. Now let 
\begin{equation*}
  M(t) = \sum_{j=0}^4 \(\|A_ju\|_{\ell^p_{0\le\gamma\lesssim t}L^q(I_\gamma;L^r)} +
\|A_ju\|_{L^\infty([0,t];L^2)}\). 
\end{equation*}
It satisfies
\begin{equation*}
  M(t)\lesssim \|u_0\|_{\Sigma} + M(t)^{1+2\si\delta},
\end{equation*}
with $\delta$ as above. Recall that standard bootstrap argument:  
\begin{lemma}[Bootstrap argument]\label{lem:bootstrap}
Let $M=M(t)$ be a nonnegative continuous function on $[0,T]$ such
that, for every $t\in [0,T]$, 
\begin{equation*}
  M(t)\le \eps_1 + \eps_2 M(t)^\kappa,
\end{equation*}
where $\eps_1,\eps_2>0$ and $\kappa >1$ are constants such that
\begin{equation*}
  \eps_1 <\left(1-\frac{1}{\kappa} \right)\frac{1}{(\kappa \eps_2)^{1/(\kappa
-1)}}\ ,\ \ \ M(0)\le  \frac{1}{(\kappa \eps_2)^{1/(\kappa-1)}}.
\end{equation*}
Then, for every $t\in [0,T]$, we have
\begin{equation*}
  M(t)\le \frac{\kappa}{\kappa -1}\ \eps_1.
\end{equation*}
\end{lemma}
We
infer that $M$ is bounded for all $t\ge 
0$ provided that $\|u_0\|_{\Sigma}$ is sufficiently small. Once we
know that $A_ju\in L^\infty([0,\infty);L^2)$ for all $j\in
\{0,\dots,4\}$, we readily infer, by using Strichartz estimates again, 
  that $\(e^{itH}A_j(t)u\)_{t\ge 0}$ is a Cauchy sequence as $t\to
  \infty$, hence the  result.

\section{Anisotropic Morawetz estimates}
\label{sec:morawetz}

In this section, we prove anisotropic Morawetz estimates, as stated in
Proposition~\ref{prop:morawetz}. 
\begin{proof}
We proceed as in \cite{GiVe10}, by working on quadratic quantities
which appear naturally in the hydrodynamical reformulation of
\eqref{eq:NLSP}. We deduce a monotonicity formula for an appropriate
virial quantity and then we infer the a priori bounds for the desired
quantities. To shorten the notations, we set $z=(x, y)$. In
\cite{GiVe10} the calculations rely on the conservation laws for the
first two momenta related to the wave function: if $u$ is a solution
to \eqref{eq:NLSP}, then we have 
\begin{equation}\label{eq:cons_laws}
\left\{\begin{aligned}
&\d_t\rho+\diver J=0\\
&\d_tJ+\diver\(\RE(\nabla\bar u\otimes\nabla u)\)+\lambda\frac{\sigma}{\sigma+1}\nabla\rho^{\sigma+1}
+\rho\nabla V=\frac14\nabla\Delta\rho,
\end{aligned}\right.
\end{equation}
where $\rho(t, z):=|u(t, z)|^2$ and $ J(t, z):=\IM(\bar u\nabla u)(t, z)$.
Let us define the virial potential
\begin{equation*}
I(t):=\frac12\iint_{\R^d\times\R^d}\rho(t, z)a(|z-z'|)\rho(t, z')\,dzdz'=\frac12\langle\rho, a\ast\rho\rangle,
\end{equation*}
where $a$ is a sufficiently smooth weight function which will be
chosen later. Here $\langle\cdot,\cdot\rangle$ denotes the scalar
product in $L^2(\R^d)$. By using \eqref{eq:cons_laws}, we see that the time
derivative of $I(t)$ reads 
\begin{equation}\label{eq:mor_act}
\frac{d}{dt}I(t)=-\langle\rho,\nabla a\ast J\rangle=\iint\rho(t, z')\nabla a(|z-z'|)\cdot J(t, z)\,dz'dz=:M(t),
\end{equation}
where $M(t)$ is the Morawetz action. By using again the balance laws
\eqref{eq:cons_laws} we have 
\begin{equation*}
\begin{aligned}
\frac{d}{dt}M(t)=&-\langle J,\nabla^2a\ast J\rangle+\langle\rho,\nabla^2a\ast\RE(\nabla\bar u\otimes\nabla u)\rangle
+\frac{\lambda\sigma}{\sigma+1}\langle\rho,\Delta a\ast\rho^{\sigma+1}\rangle\\
&-\langle\rho,\nabla a\ast(\rho\nabla V)\rangle-\frac14\langle\rho,\Delta a\ast\Delta\rho\rangle\\
=&-\langle\IM(\bar u\nabla u), \nabla^2a\ast\IM(\bar u\nabla u)\rangle
+\langle\rho,\nabla^2a\ast(\nabla\bar u\otimes\nabla u)\rangle\\
&+\frac{\lambda\sigma}{\sigma+1}\langle\rho,\Delta a\ast\rho^{\sigma+1}\rangle
-\langle\rho,\nabla a\ast(\rho\nabla V)\rangle
-\frac14\langle\rho,\Delta a\ast\Delta\rho\rangle,
\end{aligned}
\end{equation*}
where in the second term we dropped the real part because of the
symmetry of $\nabla^2a$. For the first term we have 
\begin{equation*}
\begin{aligned}
\langle\IM(\bar u\nabla u),\nabla^2a\ast\IM(\bar u\nabla u)\rangle
=&\langle\bar u\nabla u, \nabla^2a\ast(\bar u\nabla u)\rangle
-\langle\RE(\bar u\nabla u),\nabla^2a\ast\RE(\bar u\nabla u)\rangle\\
=&\langle\bar u\nabla u, \nabla^2a\ast(\bar u\nabla u)\rangle
-\frac14\langle\nabla\rho,\nabla^2a\ast\nabla\rho\rangle,
\end{aligned}
\end{equation*}
hence we may write
\begin{equation}\label{eq:dtm}
\begin{aligned}
\frac{d}{dt}M(t)=&-\langle\bar u\nabla u,\nabla^2a\ast(\bar u\nabla u)\rangle
+\langle\rho, \nabla^2a\ast(\nabla\bar u\otimes\nabla u)\rangle\\
&+\frac14\langle\nabla\rho,\nabla^2a\ast\nabla\rho\rangle
-\frac14\langle\rho,\Delta a\ast\Delta\rho\rangle\\
&+\frac{\lambda \sigma}{\sigma+1}\langle\rho, \Delta
a\ast\rho^{\sigma+1}\rangle\\
&
-\iint \rho(t,z')\rho(t,z)\nabla a(|z-z'|)\cdot \nabla V(z)dzdz'.
\end{aligned}
\end{equation}
Let us consider the second line in \eqref{eq:dtm}. It is equal to
\begin{equation*}
\frac14\langle\nabla\rho,\nabla^2a\ast\nabla\rho\rangle
-\frac14\langle\rho,\Delta a\ast\Delta\rho\rangle
=\frac12\langle\nabla\rho, \nabla^2a\ast\nabla\rho\rangle
=\frac12\langle\nabla\rho, \Delta a\ast\nabla\rho\rangle.
\end{equation*}
Now, we notice the first two terms on the right hand side of
\eqref{eq:dtm} may be rewritten as 
\begin{multline*}
-\langle\bar u\nabla u, \nabla^2a\ast(\bar u\nabla u)\rangle
+\langle\rho, \nabla^2a\ast(\nabla\bar u\otimes\nabla u)\rangle=\\
\frac12\iint_{\R^d\times\R^d}\nabla^2a(|z-z'|)\cdot(\bar u(t, z')\nabla\bar u(t, z)-\bar u(t, z)\nabla\bar u(t, z'))\\
\cdot(u(t, z')\nabla u(t, z)- u(t, z)\nabla u(t, z'))\,dz'dz,
\end{multline*}
which is non-negative if the Hessian matrix $\nabla^2a$ is non-negative definite. On the other hand, by choosing the weight 
$a(|z|)$ depending only on $y\in\R^{d-n}$ the last term in
\eqref{eq:dtm} involving the potential $V=V(x)$, vanishes. Therefore,
the natural choice for the weight is $a(|z|)=|y|$. In this way we have 
\begin{equation}\label{eq:dtm2}
\begin{aligned}
\frac{d}{dt}M(t)\ge &\frac12\langle\nabla_y\rho,
\Delta_ya\ast\nabla_y\rho\rangle 
+\frac{\lambda\sigma}{\sigma+1}\langle\rho, \Delta_ya\ast\rho^{\sigma+1}\rangle.
\end{aligned}
\end{equation}
Notice that $\Delta_ya(|y|)=\frac{d-n-1}{|y|}$ is, up to a multiplicative constant, the integral kernel of the operator 
$(-\Delta_y)^{-\frac{d-n-1}{2}}$, that is,
\begin{equation*}
\((-\Delta_y)^{-\frac{d-n-1}{2}}f\)(y)=\int_{\R^{d-n}}\frac{c}{|y-y'|}f(y')\,dy'.
\end{equation*}
Thus, by recalling $z=(x, y)$, we obtain
\begin{multline*}
\iint_{\R^d\times\R^d}\frac{1}{|y-y'|}\nabla_y\rho(t, z')\cdot\nabla_y\rho(t, z)\,dz'dz\\
=\iiint_{\R^n\times\R^n\times\R^{d-n}}\nabla_y\rho(t, x, y)\cdot\nabla_y(-\Delta_y)^{-\frac{d-n-1}{2}}\rho(t, x', y)\,dxdx'dy.
\end{multline*}
Hence, if we define the marginal of the mass density
\begin{equation*}
R(t, y):=\int_{\R^n}\rho(t, x, y)\,dx,
\end{equation*}
the last integral also reads
\begin{equation*}
\int_{\R^{d-n}}\left||\nabla_y|^{\frac{3-(d-n)}{2}}R(t, y)\right|^2\,dy.
\end{equation*}
We now plug this expression into \eqref{eq:dtm2} and we integrate in time. Furthermore, since $\lambda>0$ the second term in the right hand side in \eqref{eq:dtm2} is positive. We then infer
\begin{equation}\label{eq:mor_est}
\int_0^T\int_{\R^{d-n}}\left||\nabla_y|^{\frac{3-(d-n)}{2}}R(t,
  y)\right|^2\,dydt\le C\sup_{t\in[0, T]}M(t).
\end{equation}
Furthermore, with our choice of the weight $a(|z|)$, we have
\begin{equation*}
M(t)=\iint\rho(t, z')\frac{y-y'}{|y-y'|}\cdot\IM(\bar u\nabla_yu)(t, z)\,dz'dz
\le\|u_0\|_{L^2(\R^d)}^3\|\nabla_yu(t)\|_{L^2(\R^d)},
\end{equation*}
where we have used the conservation of the $L^2$-norm of $u$ ($V$ is
real-valued). We infer that the right hand side in \eqref{eq:mor_est}
is uniformly bounded, 
\begin{equation*}
\int_0^T\int_{\R^{d-n}}\left||\nabla_y|^{\frac{3-(d-n)}{2}}R(t,
  y)\right|^2\,dydt\le C\|u_0\|_{L^2(\R^d)}^3
\sup_{t\in\R}\|\nabla_yu(t)\|_{L^2(\R^d)}.
\end{equation*}
By letting $T$ go to infinity, we prove \eqref{eq:mor_mar}.
\end{proof}
\begin{remark}
Let us come back to \eqref{eq:dtm2} in the proof of
Proposition~\ref{prop:morawetz}. We see that the second term in the
right hand side is nonnegative, and is equal to
\begin{multline*}
\frac{\lambda\sigma}{\sigma+1}\iint\rho(t,
z')\frac{d-n-1}{|y-y'|}\rho^{\sigma+1}(t, z)\,dz'dz\\ 
\le\frac{\lambda\sigma}{\sigma+1}\iint\frac{d-n-1}{|y-y'|}\rho^{\frac{\sigma+2}{2}}(t, z')\rho^{\frac{\sigma+2}{2}}(t, z)\, dz'dz.
\end{multline*}
As before, $\frac{1}{|y-y'|}$ is the integral kernel of $(-\Delta_y)^{-\frac{d-n-1}{2}}$, hence the last integral may be written in the following way,
\begin{equation*}
C\int_{\R^{d-n}}\left||\nabla_y|^{-\frac{d-n-1}{2}}\int_{\R^n}\rho^{\frac{\sigma+2}{2}}(t, x, y)\,dx\right|^2\,dy.
\end{equation*}
Consequently, from Morawetz estimates we also infer the following a priori bound
\begin{equation*}
\iint_{\R\times\R^{d-n}}\left||\nabla_y|^{-\frac{d-n-1}{2}}\int\rho^{(\sigma+2)/2}(t, x, y)\,dx\right|^2\,dydt
\le C\|u_0\|_{L^2}^3\sup_{t\in\R}\|\nabla_yu(t)\|_{L^2}.
\end{equation*}
This term does not appear in the statement of
Proposition~\ref{prop:morawetz},  since \eqref{eq:mor_mar} is the
only estimate we are going to use to prove asymptotic completeness. 
\end{remark}

\section{Asymptotic completeness}
\label{sec:ac}

\subsection{Outline of the proof and technical remarks}
\label{sec:outline}
The strategy of the proof of Theorem~\ref{theo:ac} is the same as in
e.g. \cite{PlVe09} (Section~5.1, for the one-dimensional example), or
\cite{BCD09}. Global in time estimates for the solution $u$ are
obtained by an inductive bootstrap argument:
\begin{lemma}\label{Global bootstrap argument}\label{lem:bootstrap-global}
  Let $t_0\in \R$, $h\in L^q_{\rm loc}([t_0,\infty))$ for some $1\le
  q\le \infty$. Suppose that
  there exist $C_0$, $\eta,\kappa>0$,  and $f\in
  L^p([t_0,\infty))$ with  $1\le p<\infty$, such that for all $t'\ge t\ge t_0$,
  \begin{equation}
    \label{eq:increment}
    \|h\|_{L^q([t,t'])}\le C_0 + \|f\|_{L^p([t,t'])}^\eta \|h\|_{L^q([t,t'])}^\kappa.
  \end{equation}
Then $h$ is globally integrable,
\begin{equation*}
   h\in L^q([t_0,\infty)).
\end{equation*}
\end{lemma}
\begin{proof}
  If $\kappa>1$, split $[t_0,\infty)$ into finitely many sub-intervals $I_j$,
 on which
  $\|f\|_{L^p(I_j)}^\eta\le \eps_2$, with
  $\eps_2=\eps_2(\eta,f,C_0,\kappa)$ so small that
  Lemma~\ref{lem:bootstrap} implies  
  \begin{equation*}
   \|h\|_{L^q(I_j)} \le \frac{\kappa}{\kappa-1}C_0,
  \end{equation*}
and the result follows. If $\kappa\le 1$, we may invoke Young
inequality $ab\lesssim a^q+b^{q'}$, $1<q<\infty$, to fall back into
the first case.
\end{proof}
\begin{remark}
  When $\kappa>1$, it is crucial in the above lemma that $C_0$ is independent of
  $t,t'$. If we suppose for instance that we have inequalities of the form
 \begin{equation*}
    M(t')\le M(t) + \|f\|_{L^p([t,t'])}^\eta M(t')^\kappa,
  \end{equation*}
then $M$ need not be bounded. Consider for instance $M(t)=t$: for
$t'\ge t\ge 1$,
\begin{equation*}
  t' =t+t'-t = t +\frac{t'-t}{tt'}tt'\le t
  +\frac{t'-t}{tt'}(t')^2,
\end{equation*}
that is,
\begin{equation*}
 M(t')\le M(t)+\(\int_{t}^{t'}\frac{dt}{t^2}\)M(t')^2. 
\end{equation*}
\end{remark}
We also need a Gronwall type argument:
\begin{lemma}\label{lem:Gronwall}
  Let $t_0\in \R$, $M\in C([t_0,\infty))$. Suppose that
  there exist $\eta>0$,  and $f\in
  L^p([t_0,\infty))$ with  $1\le
  p<\infty$, such that for all $t'\ge t\ge t_0$, 
  \begin{equation}
    \label{eq:increment2}
    M(t')\le M(t) + \|f\|_{L^{p}([t,t'])}^\eta  M(t').
  \end{equation}
Then $M$ is bounded, $M\in L^\infty([t_0,\infty))$.
\end{lemma}
\begin{proof}
  Split $[t_0,\infty)$ into finitely many sub-intervals $I_j$,
  \begin{equation*}
    [t_0,\infty)=\bigcup_{j=0}^NI_j,\quad I_j=[t_j,t_{j+1}),\ t_{N+1}=\infty,
  \end{equation*}
 on which
  $\|f\|_{L^p(I_j)}^\eta \le 1/2$. Then 
  \begin{equation*}
    \sup_{t\in I_j}M(t)\le 2 M(t_j),
  \end{equation*}
hence
\begin{equation*}
  \sup_{t\ge t_0}M(t)\le 2^{N+1}M(t_0).
\end{equation*}
\end{proof}
Now the goal is to use Strichartz inequalities and H\"older inequality
in order to obtain \eqref{eq:increment} with $h(t)=
\|u(t)\|_{Z}$ for some suitable $Z$, so that the existence of
such a function $f$ follows from the conservation of energy and
Morawetz estimates. Lemma~\ref{lem:Gronwall} is then applied to
$M(t)=\|Bu\|_{L^p([0,t];Z)}$, where $B\in
\{A_1,\dots,A_4\}$. Essentially, the nonlinear estimates of
Lemma~\ref{lem:bootstrap-global} become linear estimates as in
Lemma~\ref{lem:Gronwall} because all the vector-fields $A_j$ act like
first order derivatives on gauge invariant nonlinearities. Note that
in view of the conservation of energy (see
Proposition~\ref{prop:Cauchy}) and the identity
\eqref{eq:energy-rotation}, we may replace $h$ at the first step with
\begin{equation*}
  h(t) = \sum_{j=0}^3 \|A_j(t)u\|_{Z},
\end{equation*}
since the conservations of mass and energy yield an a priori bound in 
\begin{equation*}
  \Sigma_{\rm anis} = \{f\in L^2(\R^d),\quad xf,\nabla_x f,\nabla_y f\in
  L^2(\R^d)\}. 
\end{equation*}
In that case, only the operator $A_4$, which contains the large time decay
information, is left out. We will follow both approaches below. 
\smallbreak

At this stage, the strategy may appear relatively standard. The less
standard aspect is that the numerology associated to Strichartz
estimates in the present framework does not allow to follow directly
the above road map by working in $\ell^p_\gamma L^q(I_\gamma,L^r)$,
with $(p,q,r)$ as in Theorem~\ref{theo:strichartz}. Indeed, 
apply H\"older inequality to the quantity
\begin{equation*}
  \|\lvert u\rvert^{2\si}u\|_{\ell_\gamma^{p_1'}L^{q_1'}L^{r_1'}},
\end{equation*}
which appears after the use of inhomogeneous Strichartz estimates, to
bound it by 
\begin{equation*}
   \|u\|_{\ell_\gamma^{\theta}L^{\omega}L^{s}}^{(2\si+1)(1-\eta)}
\|u\|^{(2\si+1)\eta}_{\ell_\gamma^{p}L^{q}L^{r}}, 
\end{equation*}
where $\eta\in (0,1)$, and $(p_1,q_1,r_1)$, $(p,q,r)$ are triplets as in
Theorem~\ref{theo:strichartz}, with
\begin{align*}
  &1=\frac{1}{p_1}+\frac{(2\si+1)(1-\eta)}{\theta}+
  \frac{(2\si+1)\eta}{p},\\
&1=\frac{1}{q_1}+\frac{(2\si+1)(1-\eta)}{\omega}+
  \frac{(2\si+1)\eta}{q},\\
&1=\frac{1}{r_1}+\frac{(2\si+1)(1-\eta)}{s}+
  \frac{(2\si+1)\eta}{r}.
\end{align*}
Then since $d-n<d$, we necessarily have
$  \theta<\omega,$
unless $r_1=r=2$, a case where Strichartz estimates are not even
needed (energy estimate). Unless we proceed as in
Section~\ref{sec:wave}, which means unless we use Gagliardo-Nirenberg
inequality, we cannot claim that $u \in
\ell_\gamma^{\theta}L^{\omega}L^{s}$ with $\theta<\omega$, by invoking
only Morawetz estimates and the conservation of energy. To overcome
this issue, we shall use the Strichartz estimates from
Proposition~\ref{prop:generalization} (which, in our case, follow
essentially from \cite{TzVi12}, since the harmonic oscillator possesses
an eigenbasis): the above space $Z$ will be of the form $L^r_yL^2_x$,
or a more sophisticated version of it. 
 We prove Theorem~\ref{theo:ac} from the easiest case
to the most involved technically: if $d=3$ or $4$, a rather
straightforward estimate allows to apply
Lemmas~\ref{lem:bootstrap-global} and \ref{lem:Gronwall},
successively. When $d=2$, the same strategy works for $2<\si<4$: since
this does not cover the whole range of values for $\sigma$, we
directly present another approach, more in the spirit of
\cite{PTV-p}. 

\subsection{The case $d=4$}
\label{sec:ac4}

In addition to the global existence result stated in
Proposition~\ref{prop:Cauchy}, Proposition~\ref{prop:morawetz} yields
the information $u \in L^4_t L^4_{y}L^2_x$. Let 
$B\in \{{\rm Id}, A_1,A_2,A_3\}$. As recalled above, $Bu\in L^\infty_t
L^2_{x,y}$ from the conservation of the energy (see also
\eqref{eq:energy-rotation}). 
In view of Proposition~\ref{prop:generalization} and H\"older
inequality, we have, since $(\frac{10}{3},\frac{10}{3})$ is
$3$-admissible, for $t\in I$ any interval
\begin{align*} 
 \|Bu\|_{L^{10/3}_t L^{10/3}_y L^2_x} &\lesssim \|u_0\|_{\Sigma} + \|
 |u|^{2\si}Bu\|_{L^{10/7}_t L^{10/7}_y L^2_x} \\
& \lesssim \|u_0\|_{\Sigma} + \|
 u\|_{L^{5\si}_t L^{5\si}_y L^\infty_x}^{2\si}
 \|Bu\|_{L^{10/3}_t L^{10/3}_y L^2_x} .
\end{align*}
Since $n=1$, we invoke the standard estimate
\begin{equation*}
  \|f\|_{L^\infty(\R)}\le \sqrt 2 \|f\|^{1/2}_{L^2(\R)}\|\d_x
  f\|^{1/2}_{L^2(\R)},
\end{equation*}
to infer
\begin{equation*}
 \| u\|_{L^{5\si}_t L^{5\si}_y L^\infty_x} \lesssim 
\|u\|_{L^p_t L^p_y L^2_x}^{1/2}\|u\|^{1/2}_{L^{10/3}_tL^{10/3}_y
  H^1_x},\quad \text{where }\frac{1}{p}= \frac{4-3\si}{10\si}. 
\end{equation*}
For $\si\in (\frac{2}{3},1)$, we have
\begin{equation*}
  \frac{1}{10}<\frac{1}{p}<\frac{3}{10},
\end{equation*}
thus we can always interpolate
\begin{equation*}
  \|u\|_{L^p_t L^p_y L^2_x} \le \|u\|_{L^4_tL^4_yL^2_x}^\alpha
  \|u\|_{L^q_tL^q_y L^2_x}^{1-\alpha},
\end{equation*}
with $0<\alpha<1$, and where $q$ is either $10$ or $\frac{10}{3}$. The
value $10$ is motivated by the fact that the pair $(10,10)$ in
$3$-admissible at the level of $\dot
H^1$ (instead of $L^2$ so far), and so 
\begin{equation*}
  \|u\|_{L^{10}_tL^{10}_y L^2_x}\lesssim \|\nabla_y u\|_{L^{10}_tL^{30/13}_y L^2_x},
\end{equation*}
where the pair $(10,\frac{30}{13})$ is $3$-admissible. Letting
\begin{equation*}
h(t)= \sup_{(p,r)\
  3-\text{admissible}}\sum_{j=0}^3\|A_j(t)u\|_{L^r_y L^2_x}^p, 
\end{equation*}
we conclude that this function satisfies, for all interval $I$,
\begin{equation*}
  \|h\|_{L^1(I)}\lesssim 1+ \|u\|_{L^4(I;L^4_yL^2_x)}^\beta \|h\|_{L^1(I)}^\kappa,
\end{equation*}
for some $\beta,\kappa>0$. Lemma~\ref{lem:bootstrap-global} yields
\begin{equation*}
  u,A_1u,A_2u,A_3u \in L^p(\R;L^r_yL^2_x)
\end{equation*}
for all $3$-admissible pairs $(p,r)$. Now taking $B=A_4$ in the above
estimate, Lemma~\ref{lem:Gronwall} yields $A_4u \in
L^p(\R;L^r_yL^2_x)$. Theorem~\ref{theo:ac} follows for $d=4$, thanks
to the classical arguments recalled in Section~\ref{sec:wave}. 
\subsection{The case $d=3$}
\label{sec:ac3}

We now have the a priori information $|\nabla_y|^{1/2}\(|u|^2\)\in
L^2_tL^2_yL^1_x$. In view of Sobolev embedding, we infer
\begin{equation*}
  \|u\|_{L^4_tL^8_yL^2_x}^2 = \||u|^2\|_{L^2_tL^4_yL^1_x}\lesssim
  \left\| |\nabla_y|^{1/2}\(|u|^2\)\right\|_{L^2_tL^2_yL^1_x}<\infty. 
\end{equation*}
We then proceed as in the case $d=4$. Notice that $(4,4)$ is
$2$-admissible. For $B\in \{{\rm Id}, A_1,A_2,A_3\}$, 
\begin{align*} 
 \|Bu\|_{L^{4}_t L^{4}_y L^2_x} &\lesssim \|u_0\|_{\Sigma} + \|
 |u|^{2\si}Bu\|_{L^{4/3}_t L^{4/3}_y L^2_x} \\
& \lesssim \|u_0\|_{\Sigma} + \|
 u\|_{L^{4\si}_t L^{4\si}_y L^\infty_x}^{2\si}
 \|Bu\|_{L^{4}_t L^{4}_y L^2_x}. 
\end{align*}
We interpolate the nonlinear potential by
\begin{equation*}
  \|u\|_{L^{4\si}_t L^{4\si}_y L^\infty_x} \lesssim
  \|u\|^\theta_{L_t^{\frac{8\si}{2-\si}}L_y^{\frac{8\si}{2-\si}}L^2_x} 
\|u\|^{1-\theta}_{L^4_tL^4_y H^1_x},\quad \theta= \frac{2(\sigma-1)}{3\si-2}, 
\end{equation*}
and we note that $0<\theta<1/2$ since $1<\si<2$. We interpolate again to
introduce the quantity controlled by Morawetz estimate,
\begin{equation*}
  \|u\|_{L_t^{\frac{8\si}{2-\si}}L_y^{\frac{8\si}{2-\si}}L^2_x} \le 
\|u\|_{L^\infty_t L^{\frac{8(3\si-2)}{2-\si}}_y L^2_x}^{1-\alpha}
\|u\|_{L^4_tL^8_yL^2_x}^{\alpha},\quad \alpha = \frac{2-\si}{2\si}\in (0,1). 
\end{equation*}
Notice that $y\in \R^2$, so by Sobolev embedding,
\begin{equation*}
  \|u\|_{L^\infty_t L^{\frac{8(3\si-2)}{2-\si}}_y L^2_x}\lesssim
\|u\|_{L^\infty_t H^1_y L^2_x}\lesssim \|u_0\|_{\Sigma},
\end{equation*}
where the last inequality stems from Proposition~\ref{prop:Cauchy}. 
We can then conclude as in the case $d=4$, by considering first
\begin{equation*}
  h(t) = \sum_{j=0}^3\|A_j(t)u\|_{L^4_y L^2_x}^4,
\end{equation*}
hence Theorem~\ref{theo:ac} in the case $d=3$. 
\subsection{The case $d=2$}
\label{sec:ac2}

In this case, instead of working mainly with an $L^2$ regularity in
$x$, we consider a Banach algebra. Since $x$ is associated to a
harmonic oscillator, it is not sensible to require only a Sobolev type
regularity of the form $H^s_x$, since the harmonic oscillator rotates
the phase space (see e.g. \cite{CaCauchy}). Instead, we consider the
domain of the fractional harmonic oscillator: for $s\ge 0$ a real
number, let
\begin{equation*}
  \Sigma^s_x = \left\{ f\in L^2(\R),\quad
  \|f\|_{\Sigma^s_x}:=\|f\|_{L^2(\R)}+\left\|
    \(-\d_x^2+x^2\)^{s/2}f\right\|_{L^2(\R)}<\infty\right\}. 
\end{equation*}
We know from \cite[Theorem~2.1]{BCM08} that (for any nonnegative real
number $s$) 
the above norm enjoys the following equivalence,
\begin{equation*}
  \|f\|_{\Sigma^s_x}\sim \|f\|_{L^2(\R)} + \|f\|_{\dot H^{s}(\R)} +
  \|\lvert x\rvert^{s}f\|_{L^2(\R)}.
\end{equation*}
In particular, $\Sigma^s_x$ is a Banach algebra as soon as $s>1/2$. 
\begin{lemma}
    Let $d\ge 2$, $1\le n\le d-1$, and $0\le s\le 1$.
Then for all $u_0\in L^2_y(\R^{d-n};\Sigma_x^s)$, all $(d-n)$-admissible pairs
$(p_1,r_1)$ and $(p_2,r_2)$, there exists $C_{r_1,r_2}$ such that for all
interval $I\ni 0$, the
solution to 
\begin{equation*}
   i\d_t u  =Hu +F, \quad u_{\mid t=0}=u_0,
\end{equation*}
satisfies:
  \begin{align*}
    \|u\|_{L^{p_1}_t(I; L^{r_1}_y \Sigma^s_x)}  \le
    C_{r_1,r_2}\(\|u_0\|_{L^2_y\Sigma_x^s}+ \|F\|_{L^{p_2'}_t (I;L^{r_2'}_y 
      \Sigma^s_x)}\). 
  \end{align*}
\end{lemma}
\begin{proof}
  For $s=0$, the result is a particular case of
  Proposition~\ref{prop:striTV} (or
  Proposition~\ref{prop:generalization}). For $s=1$, apply the
  vector-fields $A_1$ and $A_2$ to obtain the result from the case
  $s=0$ and the identity \eqref{eq:energy-rotation}. The case $0<s<1$
  follows by interpolation. 
\end{proof}
Morawetz estimate yields $|u|^2\in L^2_t \dot H^1_y L^1_x$. In view of
the one-dimensional inequality
\begin{equation*}
  \|f\|_{L^\infty_y} \lesssim \|f\|_{L^1_y}^{1/3}\|\d_y f\|_{L^2_y}^{2/3},
\end{equation*}
we infer the analogue of the property established in \cite{PlVe09},
$u\in L^6_t L^\infty_y L^2_x$. 

On the other hand, the conservation of mass and energy yields
\begin{equation*}
  u\in L^\infty_t H^{1/2-\alpha}_y \Sigma_x^{1/2+\alpha},\quad 
  0\le \alpha\le \frac{1}{2}.
\end{equation*}
Sobolev embedding yields ($d-n=1$)
\begin{equation*}
  u\in L^\infty_t L^{1/\alpha}_y \Sigma_x^{1/2+\alpha},\quad 
  0\le \alpha\le \frac{1}{2}.
\end{equation*}
By interpolation, we infer
\begin{equation*}
  u\in L^{6/\alpha}_t L^{1/(\alpha(1-\alpha))}_y
  \Sigma_x^{1/2+\delta},\quad 0<\alpha<\frac{1}{2},\quad \delta =
  \(1-\alpha\)\(\frac{1}{2}+\alpha\)-\frac{1}{2}>0. 
\end{equation*}
For $\alpha\in (0,1/2)$ to be fixed, and $(p,r)$ a $1$-admissible pair
to fix too, write
\begin{equation*}
  \| |u|^{2\si}u\|_{L^{p'}_t L^{r'}_y \H_x}\lesssim
  \|u\|^{(2\si+1)\eta}_{L^p_t L^r_y \H_x} \|u\|^{(2\si+1)(1-\eta)}_{
    L^{6/\alpha}_t L^{1/(\alpha(1-\alpha))}_y \H_x},
\end{equation*}
where $\H_x = \Sigma_x^{1/2+\delta}$, and $\eta\in (0,1)$. We have
used H\"older inequality, with 
\begin{align}
  1&= \frac{(2\si+1)\eta+1}{p} + (2\si+1)(1-\eta)\frac{\alpha}{6} \notag\\
1&= \frac{(2\si+1)\eta+1}{r} + (2\si+1)(1-\eta)\alpha(1-\alpha). \label{eq:r}
\end{align}
Recalling that $(p,r)$ is $1$-admissible, we infer that necessarily,
\begin{equation*}
  \frac{5}{4} =
  \frac{(2\si+1)\eta}{4}+(2\si+1)(1-\eta)\alpha
\(\frac{2}{3}-\frac{\alpha}{2}\). 
\end{equation*}
Now let $0<\eps\ll 1$, and define $\eta$ by the identity
\begin{equation*}
  5 = (2\si+1)\eta+\eps . 
\end{equation*}
Since $\si>2$, for $\eps$ sufficiently small, this defines indeed
$\eta\in (0,1)$. Then, up to decreasing $\eps$, fix $0<\alpha<1/2$ so
that
\begin{equation*}
  (2\si+1)(1-\eta)\alpha
\(\frac{2}{3}-\frac{\alpha}{2}\)=\eps. 
\end{equation*}
To define all the parameters, we now recall that $r$ is given by
\eqref{eq:r}: we check that $2<r<6$, with $r\approx 6$. Once all
these parameters are fixed, Strichartz inequality yields
\begin{equation*}
  \|u\|_{L^p_t L^r_y \H_x} \lesssim \|u_0\|_{\Sigma} + \|u\|^{2\si-4+\eps}_{
    L^{6/\alpha}_t L^{1/(\alpha(1-\alpha))}_y \H_x}
\|u\|^{5-\eps}_{L^p_t L^r_y \H_x} .
\end{equation*}
Following the same argument as in other dimensions, we infer
$  u\in L^p(\R; L^r_y\H_x)$, and thus $u\in L^{p_1}(\R;L^{r_1}_y
\H_x)$ for all $1$-admissible pairs $(p_1,r_1)$. 
\smallbreak

For $B\in \{A_0,\dots,A_4\}$, we also have, on any interval
$I=[t_0,t_1]$, 
\begin{align*}
  \|B u\|_{L^p_t L^r_y L^2_x} &\lesssim \|B(t_0)u\|_{L^2} + \|u\|^{2\si-4+\eps}_{
    L^{6/\alpha}_t L^{1/(\alpha(1-\alpha))}_y L^\infty_x}
\|u\|^{4-\eps}_{L^p_t L^r_y L^\infty_x}\|B u\|_{L^p_t L^r_y L^2_x}\\
  &\lesssim \|B(t_0)u\|_{L^2} + \|u\|^{2\si-4+\eps}_{
    L^{6/\alpha}_t L^{1/(\alpha(1-\alpha))}_y \H_x}
\|u\|^{4-\eps}_{L^p_t L^r_y \H_x}\|B u\|_{L^p_t L^r_y L^2_x},
\end{align*}
where we have used the embedding $\H_x\subset L^\infty_x$. Invoking
Lemma~\ref{lem:Gronwall}, we conclude $Bu \in L^p(\R; L^r_y L^2_x)$,
hence $Bu\in L^{p_1}(\R;L^{r_1}_y
L^2_x)$ for all $1$-admissible pairs
$(p_1,r_1)$. Theorem~\ref{theo:ac} then follows in the case $d=2$.

\appendix

\section{Existence of long range effects: formal proof}
\label{sec:long}

\begin{lemma}
 Let $f\in L^2(\R^n)$ and $g\in L^2(\R^{d-n})$. Consider the solution
 $u$ to
 \begin{equation}
    \label{eq:tensor}
    i\d_t u = Hu,\quad u(0,x,y)= f(x)g(y). 
  \end{equation}
Then as $t\to +\infty$, 
\begin{equation*}
  u(t,x,y) = e^{itH_1}f(x) \otimes\frac{1}{(it)^{(d-n)/2}}\F_2
  g\(\frac{y}{t}\)e^{i\frac{|y|^2}{2t}} +o(1) \text{ in }L^2(\R^d),
\end{equation*}
where $\F_2$ stands for the partial Fourier transform with respect to
$y$. 
\end{lemma}
The lemma is straightforward, since 
\begin{equation*}
  u(t,x,y) = e^{itH_1}f(x)\otimes e^{itH_2}g(y). 
\end{equation*}
 Following the formal argument given in \cite{Ginibre},
  the above lemma suggests that for $\sigma\le 1/(d-n)$, long
  range effects are present in \eqref{eq:NLSP}. Assume $\lambda=1$ for
  simplicity. Let $u\in
  C([T,\infty);L^2(\R^d))$ be a solution of \eqref{eq:NLSP}
  such that there exists $u_+\in L^2(\R^d)$ with
  \begin{equation*}
 u_+(x,y)=f(x)g(y)\text{ and }   \left\lVert u(t)- e^{-it
     H}u_+\right\rVert_{L^2}\Tend t 
    {+\infty}0.
  \end{equation*}
Formal computations indicate that necessarily, $u_+\equiv 0$ and
$u\equiv 0$: the linear and nonlinear dynamics are no longer
comparable, due to long range effects. To see this, let $\psi=
\psi_1\otimes\psi_2$ with $\psi_1
\in C_0^\infty(\R^n)$, $\psi_2\in C_
0^\infty(\R^{d-n})$, and $t_2\ge
t_1\ge T$. By assumption,  
  \begin{equation*}
    \left\langle \psi,
      e^{it_2H}u(t_2)-e^{it_1H}u(t_1)\right\rangle= 
    -i\int_{t_1}^{t_2}\left\langle 
    e^{-itH}\psi, \left(|u|^{2\sigma}u\right)(t)\right\rangle dt
  \end{equation*}
goes to zero as $t_1,t_2\to +\infty$. The above lemma
implies that for $t\to +\infty$, we have
\begin{equation*}
  \left\langle 
    e^{-itH}\psi, \left(|u|^{2\sigma}u\right)(t)\right\rangle \approx
    \frac{1}{t^{(d-n)(\sigma+1)}} \int_{\R^n}F(t,x)dx
    \int_{\R^{d-n}}G\(\frac{y}{t}\)  dy,
\end{equation*}
for $F=e^{-itH_1} \psi_1 |e^{-itH_1}f|^{2\si}\overline{e^{-itH_1}f}$ and $G = \F_2
  \psi_2 |\F_2 g|^{2\si}\overline{\F_2 g}$. With the change of
  variable $y\mapsto t y$, the above 
  integral is equal to 
  \begin{equation*}
   \frac{1}{t^{(d-n)\sigma}} \(\int_{\R^n}F(t,x)dx\)
   \( \int_{\R^{d-n}}G(y) dy\).
  \end{equation*}
Now $F$ is periodic in time, with arbitrary mean value. But
$t\mapsto 1/t^{(d-n)\sigma}$ is not integrable, so the above quantity
is not integrable in time, unless 
$u_+= 0$. The 
conservation of mass then implies $u\equiv 0$.

\noindent {\bf Acknowledgments}. P. Antonelli and J. Drumond Silva
would like to thank the kind 
hospitality of the Institute of Mathematics at the Universit\'e
Montpellier 2, where part of this work was developed. The authors are
grateful to Nicola Visciglia for fruitful discussions.

\bibliographystyle{siam}
\bibliography{biblio}

\begin{thebibliography}{10}

\bibitem{BCD09}
{\sc V.~Banica, R.~Carles, and T.~Duyckaerts}, {\em On scattering for {NLS}:
  from {E}uclidean to hyperbolic space}, Discrete Contin. Dyn. Syst., 24
  (2009), pp.~1113--1127.

\bibitem{Barab}
{\sc J.~E. Barab}, {\em Nonexistence of asymptotically free solutions for
  nonlinear {S}chr\"odinger equation}, J. Math. Phys., 25 (1984),
  pp.~3270--3273.

\bibitem{BCM08}
{\sc N.~Ben~Abdallah, F.~Castella, and F.~M\'ehats}, {\em Time averaging for
  the strongly confined nonlinear {S}chr\"odinger equation, using almost
  periodicity}, J. Differential Equations, 245 (2008), pp.~154--200.

\bibitem{Bo93}
{\sc J.~Bourgain}, {\em Fourier restriction phenomena for certain lattice
  subsets and applications to nonlinear evolution equations {I},
  {S}chr\"odinger equations}, Geom. Funct. Anal., 3 (1993), pp.~107--156.

\bibitem{CaCCM}
{\sc R.~Carles}, {\em Linear vs. nonlinear effects for nonlinear
  {S}chr\"odinger equations with potential}, Commun. Contemp. Math., 7 (2005),
  pp.~483--508.

\bibitem{CaCauchy}
\leavevmode\vrule height 2pt depth -1.6pt width 23pt, {\em On the {C}auchy
  problem in {S}obolev spaces for nonlinear {S}chr\"odinger equations with
  potential}, Portugal. Math. (N. S.), 65 (2008), pp.~191--209.

\bibitem{Ca11}
\leavevmode\vrule height 2pt depth -1.6pt width 23pt, {\em Nonlinear
  {S}chr{\"o}dinger equation with time dependent potential}, Commun. Math.
  Sci., 9 (2011), pp.~937--964.

\bibitem{CaMi04}
{\sc R.~Carles and L.~Miller}, {\em Semiclassical nonlinear {S}chr\"odinger
  equations with potential and focusing initial data}, Osaka J. Math., 41
  (2004), pp.~693--725.

\bibitem{CazCourant}
{\sc T.~Cazenave}, {\em Semilinear {S}chr\"odinger equations}, vol.~10 of
  Courant Lecture Notes in Mathematics, New York University Courant Institute
  of Mathematical Sciences, New York, 2003.

\bibitem{CuVi11}
{\sc S.~Cuccagna and N.~Visciglia}, {\em On asymptotic stability of ground
  states of {NLS} with a finite bands periodic potential in 1{D}}, Trans. Amer.
  Math. Soc., 363 (2011), pp.~2357--2391.

\bibitem{DG}
{\sc J.~Derezi\'nski and C.~G\'erard}, {\em Scattering theory of quantum and
  classical {N}-particle systems}, Texts and Monographs in Physics, Springer
  Verlag, Berlin Heidelberg, 1997.

\bibitem{Fujiwara}
{\sc D.~Fujiwara}, {\em Remarks on the convergence of the {F}eynman path
  integrals}, Duke Math. J., 47 (1980), pp.~559--600.

\bibitem{GiBBK}
{\sc J.~Ginibre}, {\em Le probl\`eme de {C}auchy pour des {EDP}
  semi-lin\'eaires p\'eriodiques en variable d'espace}, Ast\'erisque,  (1995),
  pp.~Exp.\ No.\ 796, p.~163--187.
\newblock S\'eminaire Bourbaki, Vol.\ 1994/95.

\bibitem{Ginibre}
{\sc J.~Ginibre}, {\em An introduction to nonlinear {S}chr\"odinger equations},
  in Nonlinear waves (Sapporo, 1995), R.~Agemi, Y.~Giga, and T.~Ozawa, eds.,
  GAKUTO International Series, Math. Sciences and Appl., Gakk\={o}tosho, Tokyo,
  1997, pp.~85--133.

\bibitem{GV79Scatt}
{\sc J.~Ginibre and G.~Velo}, {\em On a class of nonlinear {S}chr\"odinger
  equations. {II} {S}cattering theory, general case}, J. Funct. Anal., 32
  (1979), pp.~33--71.

\bibitem{GiVe10}
{\sc J.~Ginibre and G.~Velo}, {\em Quadratic {M}orawetz inequalities and
  asymptotic completeness in the energy space for nonlinear {S}chr\"odinger and
  {H}artree equations}, Quart. Appl. Math., 68 (2010), pp.~113--134.

\bibitem{GrTh11}
{\sc B.~Gr{\'e}bert and L.~Thomann}, {\em K{AM} for the quantum harmonic
  oscillator}, Comm. Math. Phys., 307 (2011), pp.~383--427.

\bibitem{HaPa13}
{\sc Z.~Hani and B.~Pausader}, {\em On scattering for the quintic defocusing
  nonlinear {S}chr\"odinger equation on {${\mathbb R}\times {\mathbb T}^2$}},
  Comm. Pure Appl. Math.,  (2013).
\newblock Archived as \url{http://arxiv.org/abs/1205.6136}.

\bibitem{JosserandPomeau}
{\sc C.~Josserand and Y.~Pomeau}, {\em Nonlinear aspects of the theory of
  {B}ose-{E}instein condensates}, Nonlinearity, 14 (2001), pp.~R25--R62.

\bibitem{KRY}
{\sc L.~Kapitanski, I.~Rodnianski, and K.~Yajima}, {\em On the fundamental
  solution of a perturbed harmonic oscillator}, Topol. Methods Nonlinear Anal.,
  9 (1997), pp.~77--106.

\bibitem{KeTa98}
{\sc M.~Keel and T.~Tao}, {\em Endpoint {S}trichartz estimates}, Amer. J.
  Math., 120 (1998), pp.~955--980.

\bibitem{Ku93}
{\sc S.~B. Kuksin}, {\em Nearly integrable infinite-dimensional {H}amiltonian
  systems}, vol.~1556 of Lecture Notes in Mathematics, Springer-Verlag, Berlin,
  1993.

\bibitem{LandauQ}
{\sc L.~D. Landau and E.~M. Lifshitz}, {\em Quantum mechanics: non-relativistic
  theory. {C}ourse of {T}heoretical {P}hysics, {V}ol. 3}, Addison-Wesley Series
  in Advanced Physics, Pergamon Press Ltd., London-Paris, 1958.
\newblock Translated from the Russian by J. B. Sykes and J. S. Bell.

\bibitem{McKeTr81}
{\sc H.~P. McKean and E.~Trubowitz}, {\em The spectral class of the
  quantum-mechanical harmonic oscillator}, Comm. Math. Phys., 82 (1981/82),
  pp.~471--495.

\bibitem{Me66}
{\sc F.~G. Mehler}, {\em Ueber die {E}ntwicklung einer {F}unction von beliebig
  vielen {V}ariablen nach {L}aplaceseschen {F}unctionen h\"oherer {O}rdnung},
  J. reine und angew. {M}ath., 66 (1866), pp.~161--176.

\bibitem{Oh}
{\sc Y.-G. Oh}, {\em Cauchy problem and {E}hrenfest's law of nonlinear
  {S}chr\"odinger equations with potentials}, J. Diff. Eq., 81 (1989),
  pp.~255--274.

\bibitem{Ozawa91}
{\sc T.~Ozawa}, {\em Long range scattering for nonlinear {S}chr\"odinger
  equations in one space dimension}, Comm. Math. Phys., 139 (1991),
  pp.~479--493.

\bibitem{PiSt}
{\sc L.~Pitaevskii and S.~Stringari}, {\em Bose-{E}instein condensation},
  vol.~116 of International Series of Monographs on Physics, The Clarendon
  Press Oxford University Press, Oxford, 2003.

\bibitem{PlVe09}
{\sc F.~Planchon and L.~Vega}, {\em Bilinear virial identities and
  applications}, Ann. Sci. \'Ec. Norm. Sup\'er. (4), 42 (2009), pp.~261--290.

\bibitem{ReedSimon2}
{\sc M.~Reed and B.~Simon}, {\em Methods of modern mathematical physics.
  {I}{I}. {F}ourier analysis, self-adjointness}, Academic Press [Harcourt Brace
  Jovanovich Publishers], New York, 1975.

\bibitem{TzVi12}
{\sc N.~Tzvetkov and N.~Visciglia}, {\em Small data scattering for the
  nonlinear {S}chr\"odinger equation on product spaces}, Comm. Partial
  Differential Equations, 37 (2012), pp.~125--135.

\bibitem{PTV-p}
{\sc N.~Visciglia}.
\newblock Private communication on a joint work in progress, 2013.

\bibitem{Yafaev}
{\sc D.~Yafaev}, {\em Scattering theory: some old and new problems}, vol.~1735
  of Lecture Notes in Mathematics, Springer-Verlag, Berlin, 2000.

\bibitem{Zelditch83}
{\sc S.~Zelditch}, {\em Reconstruction of singularities for solutions of
  {S}chr\"odinger's equation}, Comm. Math. Phys., 90 (1983), pp.~1--26.

\end{thebibliography}

\end{document}